\documentclass[a4paper,12pt]{scrartcl}
\usepackage{amsmath,amssymb,latexsym,amsthm,enumerate}
\usepackage{xcolor}
\usepackage{graphicx}
\usepackage{epstopdf}
\usepackage{tikz}

\usepackage[T1]{fontenc}
\usepackage[utf8]{inputenc}
\usepackage[english]{babel}

\newcommand*{\wh}{\widehat}
\newcommand*{\wt}{\widetilde}
\newcommand*{\ol}{\overline}

\newcommand*{\eps}{\varepsilon}

\newcommand*{\N}{\mathbb{N}}
\newcommand*{\R}{\mathbb{R}}

\newcommand*{\IZ}{\mathbb{Z}}

\newcommand*{\bbN}{\mathbb N}

\newcommand*{\bbR}{\mathbb R}
\newcommand*{\bbZ}{\mathbb Z}

\newcommand*{\cF}{\mathcal{F}}
\newcommand*{\cG}{\mathcal{G}}

\newcommand*{\cY}{\mathcal{Y}}

\newcommand*{\loc}{\mathrm{loc}}
\newcommand*{\low}{\mathrm{low}}
\newcommand*{\high}{\mathrm{high}}
\newcommand*{\vertical}{\mathrm{vertical}}
\newcommand*{\sameorder}{\stackrel{o}{\sim}}

\newcommand*{\sgn}{\operatorname{sgn}}

\newcommand{\be}{\begin{eqnarray*}}
\newcommand{\ee}{\end{eqnarray*}}
\newcommand{\ben}{\begin{eqnarray}}
\newcommand{\een}{\end{eqnarray}}
\newcommand{\bi}{\begin{itemize}}
\newcommand{\ei}{\end{itemize}}

\newtheorem{theo}{Theorem}[section]

\newtheorem{lemma}[theo]{Lemma}
\newtheorem{propo}[theo]{Proposition}
\newtheorem{corollary}[theo]{Corollary}

\theoremstyle{definition}
\newtheorem{ex}[theo]{Example}

\newtheorem{remark}[theo]{Remark}

\newcounter{numpar}[section]

\title{Properties of the EMCEL scheme for approximating irregular diffusions}

\author{Stefan Ankirchner\thanks{%
Stefan Ankirchner, Institute of Mathematics, University of Jena, Ernst-Abbe-Platz 2, 07745 Jena, Germany. \emph{Email:} s.ankirchner@uni-jena.de, \emph{Phone:} +49 (0)3641 946275.}
\and Thomas Kruse\thanks{%
Thomas Kruse, Institute of Mathematics, University of Gie{\ss}en, Arndtstr.~2, 35392 Gießen, Germany.
\emph{Email:} thomas.kruse@math.uni-giessen.de, \emph{Phone:} +49 (0)641 9932102.}
\and Wolfgang L\"ohr\thanks{%
Wolfgang L\"ohr, Faculty of Mathematics, University of Duisburg-Essen, Thea-Leymann-Str.~9, 45127 Essen, Germany.
\emph{Email:} wolfgang.loehr@uni-due.de, \emph{Phone:} +49 (0)201 1834729.}
\and Mikhail Urusov\thanks{%
Mikhail Urusov, Faculty of Mathematics, University of Duisburg-Essen, Thea-Leymann-Str.~9, 45127 Essen, Germany.
\emph{Email:} mikhail.urusov@uni-due.de, \emph{Phone:} +49 (0)201 1837428.
}}

\begin{document}

\maketitle

\begin{abstract}
We prove several properties of the EMCEL scheme, which is capable of approximating one-dimensional continuous strong Markov processes in distribution on the path space (the scheme is briefly recalled).
Special cases include irregular stochastic differential equations and processes with sticky features.
In particular, we highlight differences from the Euler scheme in the case of stochastic differential equations and discuss a certain ``stabilizing'' behavior of the EMCEL scheme like ``smoothing and tempered growth behavior''.


\end{abstract}

\section*{Introduction}
The aim of this paper is to prove several desirable properties
of the EMCEL approximation scheme
(the idea behind the scheme is briefly recalled below in the introduction,
and the scheme is formally described in Section~\ref{sec:emcel}),
which is capable of approximating all
one-dimensional continuous strong Markov processes
(abbreviated as \emph{general diffusions} in what follows).

The set of general diffusions includes strong and weak solutions
of one-dimensional stochastic differential equations (SDEs)
with possibly irregular coefficients whenever uniqueness in law holds for the SDE.
For the case of SDEs, there are many different approaches
for approximating their solutions, e.g., the Euler scheme.
There are, however, many general diffusions
that cannot be written as solutions to SDEs.
This is, in particular, true for general diffusions with sticky features,
where a sticky point is located in the interior of the state space.
A related interesting phenomenon is sticky reflection,
where the sticky point is located at the boundary of the state space.
Recent years have witnessed an increased interest
in general diffusions with sticky features,
see
\cite{Bass2014},
\cite{CanCaglar2019},
\cite{EberleZimmer2019},
\cite{ep2014},
\cite{FGV2016},
\cite{GV2017},
\cite{GV2018},
\cite{HajriCaglarArnaudon:17},
\cite{KSS2011},
\cite{Konarovskyi2017},
\cite{KvR2017}
and references therein.
All such (and other) general diffusions can be approximated via the EMCEL scheme.

Let us briefly illustrate the central idea behind the construction of the EMCEL scheme. To this end let $Y=(Y_t)_{t\in [0,\infty)}$ be a general diffusion in natural scale. For simplicity we assume throughout this paragraph that the state space of the general diffusion $Y$ is equal to the whole real line $\R$. Let $(\xi_k)_{k \in \bbN}$ be an iid sequence of random variables,
on a probability space with a measure $P$, satisfying $P(\xi_k = \pm 1) = \frac12$. Given an initial value $y_0\in \R$ and a discretization parameter $h\in (0,\infty)$,
we recursively define a Markov chain
on the time grid $\{kh:k\in\bbN_0\}$
by the formula
\begin{equation}\label{eq:def_X1}
\wh X^h_0 = y_0
\quad\text{and}\quad
\wh X^h_{(k+1)h} = \wh X^h_{kh}+\wh a_h(\wh X^h_{kh}) \xi_{k+1}, \quad \text{for } k \in \bbN_0.
\end{equation}
Here the function $\wh a_h\colon \R \to [0,\infty)$,
termed \emph{scale factor} in the sequel
(as it is used to scale the incoming random variables $\xi_{k+1}$ in~\eqref{eq:def_X1}),
is chosen in such a way that 
the expected time it takes $Y$ started in $y\in \R$ to leave the interval $(y-\wh a_h(y),y+\wh a_h(y))$ is equal to $h$, i.e., $\wh a_h$ satisfies for all $y\in \R$ that $E_y[H_{y-\wh a_h(y),y+\wh a_h(y)}(Y)]=h$ where $H_{b,c}(Y)$ is the first exit time of $Y$ from the interval $(b,c)$. Next, let $\tau^h_0=0$ and then
recursively define $\tau^h_{k+1}$ as the first time $Y$ exits the interval $(Y_{\tau^h_k}-\wh a_h(Y_{\tau^h_k}), Y_{\tau^h_k}+\wh a_h(Y_{\tau^h_k}))$ after  $\tau^h_{k}$. 
It follows that the discrete-time process $(Y_{\tau^h_k})_{k\in \N_0}$ has the same law as the Markov chain $(\wh X^h_{kh})_{k\in \N_0}$ defined in~\eqref{eq:def_X1}. We say that the Markov chain $(\wh X^h_{kh})_{k \in \bbN_0}$ is embedded into $Y$ with the sequence of stopping times $(\tau_{k}^h)_{k\in \N_0}$. Moreover, the stopping times satisfy that $E_y[\tau^h_{k+1}-\tau^h_k]=h$. This explains why we refer to $(\wh X^h_{kh})_{k\in \N_0}$ as Embeddable Markov Chain with Expected time Lag $h$ and write $\wh X^h \in \text{EMCEL}(h)$ as a shorthand.
A key observation is that the requirement $E_y[H_{y-\wh a_h(y),y+\wh a_h(y)}(Y)]=h$ can be transformed into the analytic condition
$$
\int_{(y-\wh a_h(y),y+\wh a_h(y))} (\wh a_h(y)-|u-y|)\,m(du)=2h
$$
(see Remark~1.2 in \cite{aku2018cointossing}),
where $m$ denotes the speed measure of $Y$.
This condition is used to define the scale factors $\wh a_h$, $h\in(0,\infty)$,
which determine the scheme.
For more general state spaces than $\R$ this condition has to be adjusted appropriately. This is done in~\eqref{eq:08112019a1} below, where the EMCEL scheme is formally introduced in the general case.
For the discussion of the embedding stopping times
in the general situation
we refer to Section~3 and, in particular, Proposition~3.1 in \cite{aku2018cointossing}.

The properties of the EMCEL scheme discussed in this paper fall into the following three categories:
\begin{enumerate}[(i)]
\item\label{it:31012020a1}
asymptotic properties of the scale factor as $h\searrow0$ (Section~\ref{sec:cat1});
\item\label{it:31012020a2}
stability properties (Section~\ref{sec:cat2});
\item\label{it:31012020a3}
ODE characterization of the scale factor (Section~\ref{sec:cat3}).
\end{enumerate}
In category~\eqref{it:31012020a1} we, in fact, discuss many different properties.
Some of them highlight relationships and differences between the EMCEL and the Euler schemes in the SDE case.
The main result in category~\eqref{it:31012020a2} is Theorem~\ref{prop:comp_princ} (a so-called \emph{comparison principle}),
which implies some ``stabilizing'' behavior of the EMCEL scheme like ``smoothing and tempered growth behavior''.
The main result in category~\eqref{it:31012020a3} is Theorem~\ref{theo:ode} (a certain ODE for the scale factors),
which gives understanding of what regularity the scale factors do have in general
and helps implementing the EMCEL scheme in specific situations.
We refrain from a further discussion in the introduction, as such a discussion would require a thorough description of the setting.
Much more details, also to the meaning of our results, are present in Section~\ref{sec:properties}.

\bigskip
It remains to describe relations with the literature and to state our contributions.
For classical results on the approximation of SDEs with \emph{globally} Lipschitz coefficients via the Euler scheme we refer to the monographs
\cite{KloedenPlaten} and \cite{Pages:18} and to references therein.
In the case of SDEs with \emph{locally} Lipschitz coefficients the Euler scheme is known to converge almost surely (see \cite{gyongy98}),
but, unless the Lipschitz condition is global, neither in strong nor in numerically weak sense (see \cite{HJK} for the terminology and a precise result of this kind).
There are, however, results in positive direction:
e.g., see \cite{KHLY:JCAM2017} and \cite{NgoTaguchi:SPL2017}
and references therein for results on the weak and strong convergence
for the Euler-type approximations of SDEs with discontinuous coefficients.
But, in contrast to the EMCEL scheme, the Euler scheme is defined only for SDEs (not for \emph{all} general diffusions)
and may fail to converge even in the stochastically weak sense
when the SDE coefficients are irregular
(see Section~5.4 in \cite{aku-jmaa}).

It is necessary to say that there exist other schemes capable of approximating
interesting subclasses of general diffusions, e.g., SDEs with discontinuous coefficients; see
\cite{EL},
\cite{LLP2019},
\cite{LejayMartinez},
\cite{milstein2015uniform}
and references therein. However, the properties discussed in our paper are specific for the EMCEL scheme only.
The EMCEL scheme appears as an important example in \cite{aku2018cointossing}, \cite{aku2019wasserstein} and \cite{ku2019exittimes},
but the objects of study in those papers are certain classes of schemes for general diffusions.
The properties studied in the present paper 
are not shared by the mentioned classes of schemes.
Therefore, such properties are not at all discussed in the aforementioned papers.

The approach to approximate solutions of one-dimensional driftless SDEs via sequences of embedding stopping times appears in
\cite{aku-jmaa} and \cite{aku-aihp},
where \cite{aku-aihp} proves a functional limit theorem for irregular SDEs,
and \cite{aku-jmaa} elaborates a scheme for approximating irregular diffusions.
The latter is the EMCEL scheme in the SDE setting, although the name ``EMCEL'' does not appear in \cite{aku-jmaa}.
As compared to the present paper, the scheme of \cite{aku-jmaa} is described in somewhat different terms,
the discussion of the properties in \cite{aku-jmaa} is far less complete than that in the present paper
and is performed under more restrictive assumptions
(precisely: subclass of Example~\ref{ex:sde} below with locally bounded on $I^\circ$ functions $|\eta|$ and $\frac1{|\eta|}$),
which are, however, essential for the proofs in \cite{aku-jmaa}.
To be more specific, \cite{aku-jmaa} does not contain properties of category~\eqref{it:31012020a1}
(except for some very basic statements),
but it contains the mentioned comparison principle and the ODE for the scale factors in its more restrictive setting
(in particular, the ODE in \cite{aku-jmaa} is simpler).
Thus, our contribution essentially includes the properties of category~\eqref{it:31012020a1},
but also in categories \eqref{it:31012020a2} and~\eqref{it:31012020a3} our contribution in comparison to \cite{aku-jmaa}
consists not in mere generalizations of the respective proofs but rather in finding completely new ones (and the right formulation of the ODE),
as those ideas from \cite{aku-jmaa} do not work in our present generality.
More precisely, those proofs from \cite{aku-jmaa} heavily rely on the implicit function theorem, which is not applicable in our situation
(contrary to \cite{aku-jmaa}, the involved functions have kinks in general;
an it is worth noting that the latter statement also follows from
our general description in Theorem~\ref{theo:ode} below).
For further details, see Section~\ref{sec:properties}.

\section{The EMCEL scheme}\label{sec:emcel}

Let $(\Omega, \cF, (\cF_t)_{t \ge 0}, (P_y)_{y \in I}, (Y_t)_{t \ge 0})$ be a one-dimensional continuous strong Markov process in the sense of Section~VII.3 in \cite{RY}.
We refer to this class of processes as \emph{general diffusions} in the sequel. We assume that the state space is an open, half-open or closed interval $I \subseteq \R$. We denote by $I^\circ=(l,r)$ the interior of $I$, where $-\infty\leq l<r\leq \infty$, and we set $\ol I=[l,r]$.
Recall that by the definition we have $P_y[Y_0=y]=1$ for all $y\in I$. 
We further assume that $Y$ is regular. This means that for every $y\in I^\circ$ and $x\in I$ we have that $P_y[H_x(Y)<\infty]>0$, where $H_x(Y)=\inf\{t\geq 0: Y_t=x \}$ (with the usual convention $\inf\emptyset=\infty$).
Moreover, for $a<b$ in $\ol I$ we denote by $H_{a,b}(Y)$
the first exit time of $Y$ from $(a,b)$,
i.e., $H_{a,b}(Y) = H_a(Y)\wedge H_b(Y)$.
Without loss of generality we suppose that the general diffusion $Y$ is conservative (i.e., with infinite life time) and in natural scale.
While the setting in Section~VII.3 in \cite{RY} allows for finite life times, the process can be killed only at the endpoints of $I$ that do not belong to $I$,
in which case we can add such an endpoint to $I$ and make it absorbing; such a procedure gives us a conservative process.
If $Y$ is not in natural scale, then there exists a strictly increasing continuous function
$s\colon I \to \R$, the so-called scale function,
such that $s(Y)$ is in natural scale.

Let $m$ be the speed measure of the Markov process $Y$
on~$I^\circ$ (see~VII.3.7 in \cite{RY}).
Recall that for all $a<b$ in $I^\circ$ we have
\begin{equation}\label{eq:06072018a1}
0<m([a,b])<\infty.
\end{equation}
We also recall that a boundary point $b$ ($b\in\{l,r\}$)
is called accessible if $P_y[H_b(Y)<\infty]>0$ for some, hence, for all, $y\in I^\circ$
(such a definition because $Y$ is conservative).
Due to our assumption that $Y$ is regular and conservative, $b\in\{l,r\}$ is an accessible boundary if and only if $b\in I$.
For both boundaries $l$ and $r$, we assume that if a boundary point is accessible, then it is absorbing.\footnote{The remaining case
that $Y$ has at least one reflecting boundary
(both instantaneous and sticky reflections are allowed here),
is reduced to the case of inaccessible or absorbing boundaries
by a suitable symmetrization
(see Section~6 in \cite{aku2018cointossing}).
In this sense, that assumption does not result in a loss of generality.
It is only convenient for the exposition.}

For what follows, we briefly recall Feller's test for explosions
(see, e.g., Theorem~23.12 in \cite{Kallenberg2002} or Lemma~2.1 in \cite{AKKK17} or Theorem~3.3 in \cite{AKKU2018}):
the left boundary $l$ is accessible if and only if
\begin{equation}\label{eq:09042020a1}
l>-\infty\quad\text{and}\quad\int_{(l,y)}(u-l)\,m(du)<\infty
\end{equation}
for some, equivalently, for all, $y\in I^\circ$ (recall that $Y$ is in natural scale).
Symmetric statement holds of course for the right boundary point~$r$.

\begin{ex}[Driftless SDE with possibly irregular diffusion coefficient]\label{ex:sde}
Consider the case,
where inside $I^\circ$ the process $Y$ is driven by the SDE
\begin{equation}\label{eq:27092018a1}
dY_t=\eta(Y_t)\,dW_t,
\end{equation}
where $W$ is a Brownian motion and
$\eta\colon I^\circ\to\bbR$ is a Borel function
satisfying the Engelbert-Schmidt conditions
\begin{gather}
\eta(x)\ne0\;\;\forall x\in I^\circ,
\label{eq:27092018a2}\\[1mm]
\eta^{-2}\in L^1_{\loc}(I^\circ)
\label{eq:27092018a3}
\end{gather}
($L^1_{\loc}(I^\circ)$ denotes the set of Borel functions
locally integrable on~$I^\circ$).
Under \eqref{eq:27092018a2}--\eqref{eq:27092018a3}
SDE~\eqref{eq:27092018a1}
has a unique in law
(possibly exiting $I^\circ$ in finite time)
weak solution;
see \cite{ES1985} or Theorem~5.5.7 in \cite{KS}.
We make the convention that $Y$ remains constant
after reaching $l$ or $r$ in finite time,
which makes the boundary points absorbing whenever accessible.
This is a particular case of our setting,
where the speed measure of $Y$
on $I^\circ$ is given by the formula
\begin{equation}\label{eq:speed_measure_sde}
m(dx)=\frac 2{\eta^2(x)}\,dx.
\end{equation}
This completes the description in Example~\ref{ex:sde}.
\end{ex}

We now recall the EMCEL approximation scheme introduced in~\cite{aku2018cointossing}. Fix an arbitrary $\ol h\in(0,\infty)$.
For each $h \in (0, \ol h]$ we define the function $\wh a_{h}\colon \ol I \to [0,\infty)$
by the formulas $\wh a_h(l)=\wh a_h(r)=0$ and, for $y\in I^\circ$,
\begin{equation}\label{eq:08112019a1}
\wh a_h(y) = \sup\left\{a \ge 0: y\pm a \in I \text{ and } \frac{1}{2}\int_{(y-a,y+a)} (a-|u-y|)\,m(du) \le h\right\}.
\end{equation}
For now, fix some $h\in(0,\ol h]$.
We consider $h$ as a discretization parameter.
In what follows, the function $\wh a_h$ is referred to as the \emph{EMCEL scale factor}.
Notice that, for all $y\in\ol I$, we have
\begin{equation}\label{eq:08112019a3}
y\pm\wh a_h(y)\in\ol I.
\end{equation}
We next construct an approximation $\wh X^h$ of $Y$
associated to the scale factor $\wh a_h$.
To this end, we fix a starting point $y \in I^\circ$ of $Y$.
Let $(\xi_k)_{k \in \bbN}$ be an iid sequence of random variables,
on a probability space with a measure $P$, satisfying $P(\xi_k = \pm 1) = \frac12$.
We denote by $(\wh X^h_{kh})_{k \in \bbN_0}$ the Markov chain given by the formula
\begin{equation}\label{eq:def_X}
\wh X^h_0 = y
\quad\text{and}\quad
\wh X^h_{(k+1)h} = \wh X^h_{kh}+\wh a_h(\wh X^h_{kh}) \xi_{k+1}, \quad \text{for } k \in \bbN_0,
\end{equation}
which is well-defined due to~\eqref{eq:08112019a3}.
We extend $(\wh X^h_{kh})_{k \in \bbN_0}$ to a continuous-time process by linear interpolation, i.e., for all $t\in[0,\infty)$, we set
\begin{equation}\label{eq:13112017a1}
\wh X^h_t =\wh X^h_{\lfloor t/h \rfloor h} + (t/h - \lfloor t/h \rfloor) (\wh X^h_{(\lfloor t/h \rfloor +1)h} - \wh X^h_{\lfloor t/h \rfloor h}). 
\end{equation}
To highlight the dependence of $\wh X^h=(\wh X^h_t)_{t\in[0,\infty)}$ on the starting point $y\in I^\circ$ we also sometimes
write~$\wh X^{h,y}$.

The process $\wh X^h$ is referred to as Embeddable Markov Chain with Expected time Lag~$h$
(we write shortly $\wh X^h\in\text{EMCEL}(h)$).
The whole family $(\wh X^h)_{h\in (0,\ol h]}$ is referred to as the EMCEL approximation scheme.

It is worth noting that \eqref{eq:08112019a3} can be made more precise:
\begin{equation}\label{eq:09042020a2}
\text{for all }y\in I,\quad\text{we have }y\pm\wh a_h(y)\in I.
\end{equation}
This can be deduced from~\eqref{eq:08112019a1} with the help of Feller's test.
Indeed if, e.g., $y-\wh a_h(y) = l$, then $l>-\infty$ and $\int_{(z,y)}(u-z)\,m(du)\le 2h$ for all $z\in (l,y)$ (see~\eqref{eq:08112019a1}).
Letting $z\downarrow l$ we obtain~\eqref{eq:09042020a1} by monotone convergence.
Feller's test implies that, in this case, $l$ is accessible, and hence $l\in I$.
Observe that \eqref{eq:09042020a2} translates into the fact that the EMCEL scheme never reaches inaccessible boundary points of $I$
(cf.\ \eqref{eq:def_X}--\eqref{eq:13112017a1}),
which is, in fact, a desirable property for a scheme.

\begin{remark}\label{rem:02022020a1}
In more detail, what is included in~\eqref{eq:08112019a1} is as follows.
For all $y\in I^\circ$, define $a_I(y)=\min\{y-l,r-y\}$ ($\in(0,\infty]$)
and notice that, for $a\ge0$, it holds
$y\pm a\in I^\circ$ if and only if $a<a_I(y)$.
Fix $y\in I^\circ$.
It follows from~\eqref{eq:06072018a1} that
$\int_{(y-a,y+a)} (a-|u-y|)\,m(du)<\infty$ whenever $a\in[0,a_I(y))$.
Therefore, the function
$$
a\mapsto\int_{(y-a,y+a)}(a-|u-y|)\,m(du)\equiv\int_I (a-|u-y|)^+\,m(du)
$$
is strictly increasing and continuous on $[0,a_I(y))$ (by the dominated convergence theorem).
The number $\wh a_h(y)$ is thus a unique positive root of the equation (in $a\in[0,a_I(y))$)
\begin{equation}\label{eq:08112019a2}
\frac{1}{2}\int_{(y-a,y+a)} (a-|u-y|)\,m(du) = h
\end{equation}
whenever
$\sup_{a\in[0,a_I(y))}\int_{(y-a,y+a)} (a-|u-y|)\,m(du)>h$,
i.e., when $y$ is ``not too close'' to an accessible boundary point of~$I$.

We now make the last statement more precise with the help of the following notations.
If $l> -\infty$, we define, for all $h\in(0,\ol h]$,
$$
l_h = l+ \inf\left\{ a \in \left(0,\frac{r-l}2\right]:
a<\infty
\;\;\text{and}\;\;
\frac12\int_{(l, l+2a)} (a - |u-(l+a)|)\,m(du) \ge h \right\}, 
$$
where we use the convention $\inf \emptyset = \frac{r-l}2$.
If $l = -\infty$, we set $l_h = -\infty$. 
Similarly, if $r < \infty$, then we define, for all $h\in(0,\ol h]$,
$$
r_h = r- \inf\left\{ a \in \left(0,\frac{r-l}2\right]:
a<\infty
\;\;\text{and}\;\;
\frac12\int_{(r-2a, r)} (a - |u-(r-a)|)\,m(du) \ge h \right\}
$$
with the same convention $\inf \emptyset = \frac{r-l}2$.
If $r = \infty$, we set $r_h = \infty$.
In any case, it holds $l_h\le r_h$
and, moreover, $l_h\le\frac{r+l}2\le r_h$
whenever $l$ and $r$ are finite.

Using Feller's test for explosions once again we see that $l$ is inaccessible if and only if $l_h = l$ for all $h\in (0,\ol h]$.
Similary, $r$ is inaccessible if and only if $r_h = r$ for all $h\in (0,\ol h]$.
Notice that, if $l$ or $r$ are accessible, it holds that $l_h\searrow l$ or $r_h\nearrow r$, respectively, as $h\searrow 0$.

Thus, the definitions of $l_h$ and $r_h$ yield that,
for $y\in(l_h,r_h)$, the number $\wh a_h(y)$ is a unique positive root of~\eqref{eq:08112019a2},
while, for $y\in(l,l_h]$ (resp., $y\in[r_h,r)$),
$\wh a_h(y)$ is chosen to satisfy
\begin{equation}\label{eq:22022019a4}
y-\wh a_h(y)=l\qquad\text{(resp., }y+\wh a_h(y)=r).
\end{equation}
This concludes the detailed description of
what is included in~\eqref{eq:08112019a1}.
\end{remark}

We emphasize that
the EMCEL scheme is 
capable of weakly approximating \emph{every} general diffusion~$Y$.
For an illustration, we now recall a couple of results from
\cite{aku2018cointossing},
\cite{aku2019wasserstein} and
\cite{ku2019exittimes}
and refer to those papers for more detail.

To formulate the results, we need to equip $[0,\infty]\times[0,\infty]\times C([0,\infty),\bbR)$ with a suitable topology. On $[0,\infty]$ we use the topology generated by the metric
$$
d(s,t)=\left|\frac{s}{1+s}-\frac{t}{1+t} \right|,\quad s,t\in[0,\infty]
$$
with the convention that $\frac\infty\infty=1$.
We equip $C([0,\infty),\bbR)$ with the topology
of uniform convergence on compact intervals,
which is generated, e.g., by the metric
$$
\rho(x,y)=\sum_{n=1}^\infty 2^{-n}
\left(\|x-y\|_{C[0,n]}\wedge1\right),
\quad x,y\in C([0,\infty),\bbR),
$$
where $\|\cdot\|_{C[0,n]}$ denotes the sup norm
on $C([0,n],\bbR)$.
Finally, we use the standard product topology on the product space $[0,\infty]\times[0,\infty]\times C([0,\infty),\bbR)$.

The following result is a consequence of Theorem~2.1 in \cite{ku2019exittimes}.

\begin{propo}\label{prop:07022020a2}
For any speed measure $m$ and
for any $y\in I^\circ$, the distributions of the random elements
$(H_l(\wh X^{h,y}),H_r(\wh X^{h,y}),\wh X^{h,y})$ under $P$ converge weakly
to the distribution of $(H_l(Y),H_r(Y),Y)$ under $P_y$, as $h\to 0$; i.e.,
for every bounded and continuous
functional
$F\colon [0,\infty]\times[0,\infty]\times C([0,\infty),\bbR)\to \R$, it holds that
\begin{equation}\label{eq:23032019a1}
E[F(H_l(\wh X^{h,y}),H_r(\wh X^{h,y}),\wh X^{h,y})]\to E_y[F(H_l(Y),H_r(Y),Y)], \quad h\to 0.
\end{equation}
\end{propo}

We remark that the weak convergence in~\eqref{eq:23032019a1}
holds jointly for paths and \emph{exit times}
(i.e., hitting times $H_l$ and $H_r$ of the boundary points of the state space),
which is a stronger statement than the weak convergence in the path space
because the exit times are, in general, essentially discontinuous path functionals
(e.g., $H_l$ is
discontinuous with positive probability
whenever the boundary $l$ is accessible).

\medskip
The following result about convergence rates
is a consequence of Theorem~1.7 in \cite{aku2019wasserstein}.

\begin{propo}\label{prop:26022019a1}
Suppose that the speed measure $m$ satisfies
\begin{equation}\label{eq:cond_c}
m(dx)\ge \frac{2}{k(1+x^2)}dx\quad\text{on }I^\circ
\end{equation}
with some $k\in(0,\infty)$.
Let $T\in (0,\infty)$ and let $F \colon C([0,T],I) \to \R$ be a locally Lipschitz continuous path functional
with polynomially growing Lipschitz constant,
i.e., there exist $L,\alpha\in [0,\infty)$ such that
for all $x_1,x_2 \in C([0,T],I)$ it holds
\begin{equation}\label{eq:LipConstPol}
|F(x_1)-F(x_2)|\le L\left\{1+(\|x_1\|_{C[0,T]}\vee\|x_2\|_{C[0,T]})^\alpha\right\}\|x_1-x_2\|_{C[0,T]}.
\end{equation}
Then for every $\eps\in(0,\frac{1}{4})$ and $y\in I^\circ$
there exist a constant $C\in [0,\infty)$ such that
for all $h\in (0,\ol h)$ it holds
\begin{equation}\label{eq:rate_path_locLip}
\left|E\left[F(\wh X^{h,y}_{t};\,t\in[0,T])\right]
-E_y\left[F(Y_t;\,t\in[0,T])\right]\right|\
\le C
h^{\frac{1}{4}-\eps}.
\end{equation}
Moreover, in the case $F(x)=f(x(T))$  (with some $f\colon I\to\bbR$)
when the functional $F$ depends only on the terminal value $x(T)$
of $x\in C([0,T],I)$, the rate is $\frac14$, i.e.,
\eqref{eq:rate_path_locLip} holds with $\eps=0$.
\end{propo}

The role of assumption~\eqref{eq:cond_c} is to ensure that
the expected values in~\eqref{eq:rate_path_locLip} exist.
We remark that \eqref{eq:cond_c} does not exclude sticky features
mentioned above, as these are modeled via atoms in $m$.
Notice that in the case,
where $Y$ is a solution of a driftless SDE
of the form $dY_t=\eta(Y_t)\,dW_t$ (cf.\ Example~\ref{ex:sde}),
assumption~\eqref{eq:cond_c} means that $\eta$
has at most linear growth;
however, $\eta$ can be arbitrarily irregular
(just a Borel function satisfying
\eqref{eq:27092018a2}--\eqref{eq:27092018a3}).
Finally, we stress that the rate $\frac14-$
in Proposition~\ref{prop:26022019a1}
cannot be considered as too slow
because it holds in arbitrarily irregular cases
(as discussed above, the only assumption
\eqref{eq:cond_c} is not a regularity condition)
and refer to \cite{aku2019wasserstein} for more detail.

We, finally, mention that \cite{aku2019wasserstein} contains also
results about convergence rates of the EMCEL scheme
in the Wasserstein distances.

\section{Properties of the EMCEL scheme}\label{sec:properties}
In this section we gather several properties of the approximating Markov chain~\eqref{eq:def_X},
which are encoded in the functions (scale factors) $\wh a_h$, $h\in(0,\ol h]$.

\subsection{Dependence on the discretization parameter}\label{sec:cat1}
We first discuss properties from category~\eqref{it:31012020a1} of the introduction.
Specifically, here we study the
asymptotic
behavior of the EMCEL scale factors $\wh a_h$ as $h\searrow0$.
We also discuss relationships and differences with the Euler scheme in the SDE case (cf.\ Example~\ref{ex:sde}).

\begin{propo}\label{prop:07022020a1}
For any $y\in I^\circ$, the function
\begin{equation}\label{eq:02022020a1}
(0,\ol h]\ni h\mapsto\wh a_h(y)
\end{equation}
is strictly positive and nondecreasing.
Moreover, for any $y\in I^\circ$, there is $h_0(y)\in(0,\ol h]$
such that the function in~\eqref{eq:02022020a1}
is strictly increasing on $(0,h_0(y)]$.
\end{propo}

\begin{proof}
Both claims follow from the detailed description of the EMCEL scale factors $\wh a_h$, $h\in(0,\ol h]$, presented in Remark~\ref{rem:02022020a1}.
Specifically, $h_0(y)$ in the second claim can be defined as
$\sup\{h\in(0,\ol h]:y\in[l_h,r_h]\}$.
\end{proof}

\begin{propo}\label{prop:09022020a1}
For any $m\in\bbN$,
defining
$K_m=I\cap[-m,m]$, we have
\begin{equation}\label{eq:02022020a2}
\lim_{h\to0}\sup_{y\in K_m}\wh a_h(y)=0.
\end{equation}
\end{propo}

Notice that this result is a bit stronger than $\lim_{h\to0}\sup_{y\in K}\wh a_h(y)=0$ for all compact subsets $K$ of~$I^\circ$:
suprema over one-sided neighborhoods of finite boundary points of $I^\circ$ are also included in~\eqref{eq:02022020a2}.

\begin{proof}
Assume that \eqref{eq:02022020a2} does not hold for some $m\in\bbN$, i.e., there exist $\eps>0$
and sequences $\{h_n\}\subset(0,\ol h]$, $\{y_n\}\subset K_m$
such that $h_n\to0$ and $\wh a_{h_n}(y_n)\ge\eps$ for all $n$.
By considering a suitable subsequence we assume without loss of generality that
$\{y_n\}$ is monotone and $y_n\to y_\infty$ with some $y_\infty\in K_m$.
Notice that $y_\infty\in I^\circ$ and its distance from the boundary of $I^\circ$ is at least $\eps$,
otherwise \eqref{eq:08112019a3} would be violated for $y_n$ with sufficiently large $n$.
Then we get
\begin{align*}
&\liminf_{n\to\infty}\frac12\int_{(y_n-\wh a_{h_n}(y_n),y_n+\wh a_{h_n}(y_n))}(\wh a_{h_n}(y_n)-|u-y_n|)\,m(du)\\[1mm]
&\ge\liminf_{n\to\infty}\frac12\int_{(y_n-\frac12\wh a_{h_n}(y_n),y_n+\frac12\wh a_{h_n}(y_n))}\frac{\wh a_{h_n}(y_n)}2\,m(du)\\[1mm]
&\ge\frac\eps4\liminf_{n\to\infty}m\big((y_n-\eps/2,y_n+\eps/2)\big)\ge\frac\eps4 m\big((y_\infty-\eps/2,y_\infty+\eps/2)\big)>0.
\end{align*}
Observing that $\{y_n\}\subset I^\circ$ (due to $\wh a_{h_n}(y_n)>0$),
we get from~\eqref{eq:08112019a1}
$$
\frac12\int_{(y_n-\wh a_{h_n}(y_n),y_n+\wh a_{h_n}(y_n))}(\wh a_{h_n}(y_n)-|u-y_n|)\,m(du)\le h_n\to0,\quad n\to\infty.
$$
The obtained contradiction concludes the proof.
\end{proof}

Our next aim is to discuss the speed of convergence of $\wh a_h(y)$ to zero, as $h\to0$, for any fixed $y\in I^\circ$.
The next result helps establishing the order of convergence (in $h$) in many specific situations.

\begin{lemma}\label{prop:27092018a1}
For any $y\in I^\circ$, there exists $h_0\in(0,\ol h]$ such that, for $h\in(0,h_0]$, we have the inequalities
\begin{equation}\label{eq:04022020a1}
\wh a_h(y)\sup_{\lambda\in[0,1]}\left\{(1-\lambda)m([y-\lambda\wh a_h(y),y+\lambda\wh a_h(y)])\right\}
\le2h\le\wh a_h(y) m((y-\wh a_h(y),y+\wh a_h(y)))
\end{equation}
and, in particular,
\begin{equation}\label{eq:30092018a1}
\frac{\wh a_h(y)}2 m([y-\wh a_h(y)/2,y+\wh a_h(y)/2])\le2h\le\wh a_h(y) m((y-\wh a_h(y),y+\wh a_h(y))).
\end{equation}
\end{lemma}

\begin{proof}
Fix $y\in I^\circ$.
Choose a sufficiently small $h_0\in(0,\ol h]$ such that $y\in(l_{h_0},r_{h_0})$.
Remark~\ref{rem:02022020a1} implies that, for all $h\in(0,h_0]$, we have
$$
\int_{(y-\wh a_h(y),y+\wh a_h(y))}(\wh a_h(y)-|u-y|)\,m(du)=2h.
$$
For any $\lambda\in[0,1]$, we have
\begin{align*}
(1-\lambda)\wh a_h(y)m([y-\lambda\wh a_h(y),y+\lambda\wh a_h(y)])
&\le\int_{(y-\wh a_h(y),y+\wh a_h(y))}(\wh a_h(y)-|u-y|)\,m(du)\\[1mm]
&\le\wh a_h(y)m((y-\wh a_h(y),y+\wh a_h(y))).
\end{align*}
This implies both claims.
\end{proof}

We will extensively use the following terminology and notations.
Let $z\in\bbR$, $\eps>0$ and
$f,g\colon(z-\eps,z+\eps)\setminus\{z\}\to\bbR$
be two real functions defined in a deleted neighborhood of $z$.
We say that $f$ and $g$ are of the \emph{same order},
as $x\to z$, and write
\begin{equation}\label{eq:07022020a1}
f(x)\sameorder g(x),\quad x\to z,
\end{equation}
if $\limsup_{x\to z}\left|\frac{f(x)}{g(x)}\right|<\infty$
and $\liminf_{x\to z}\left|\frac{f(x)}{g(x)}\right|>0$
(with the convention $\frac00:=1$).
We also use the same notation~\eqref{eq:07022020a1}
in the case when $f$ and $g$ are only defined
in a one-sided neighborhood of $z$,
i.e., $(z,z+\eps)$ or $(z-\eps,z)$.

\begin{remark}\label{rem:27092018a1}
We fix $y\in I^\circ$ and discuss straightforward consequences
of~\eqref{eq:30092018a1} is several specific situations.

\smallskip
(a) If the speed measure has an atom at $y$,
i.e., $m(\{y\})>0$, then we get
$\wh a_h(y)\sameorder h$, $h\to0$.

\smallskip
(b) In the setting of Example~\ref{ex:sde}
with $\eta(x)\sameorder1$, $x\to y$,
we have $\wh a_h(y)\sameorder\sqrt h$, $h\to0$.

\smallskip
(c) More generally, in the setting of Example~\ref{ex:sde}
with $\eta(x)\sameorder|x-y|^\alpha$, $x\to y$,
for some $\alpha\in(-\infty,\frac12)$
(this restriction on $\alpha$ is to ensure~\eqref{eq:27092018a3}),
we obtain
$\wh a_h(y)\sameorder h^{\frac1{2-2\alpha}}$, $h\to0$.

\smallskip
It is worth noting that, by varying $\alpha\in(-\infty,\frac12)$ in~(c),
we can obtain orders $h^\beta$ for all $\beta\in(0,1)$.
\end{remark}

At this point, it is instructive to compare the EMCEL scheme and the Euler scheme.
The latter is defined only in the SDE case. More precisely, in the setting of Example~\ref{ex:sde},
we define the \emph{Euler scale factors}
$a^{Eu}_h(y)=\eta(y)\sqrt h$, $y\in I^\circ$, $h\in(0,\ol h]$
(and, to treat jumping out of the state space,
we extend the functions $a^{Eu}_h$, $h\in(0,\ol h]$, to be zero in $\bbR\setminus I^\circ$).
The (linearly interpolated, weak) Euler scheme $X^{Eu,h}=(X^{Eu,h}_t)_{t\in[0,\infty)}$
is defined through $a^{Eu}_h$ in the same way as the EMCEL scheme $\wh X^h$
is defined through $\wh a_h$ in \eqref{eq:def_X}--\eqref{eq:13112017a1}.
The properties of the Euler scheme are thus encoded in the Euler scale factors $a^{Eu}_h$, $h\in(0,\ol h]$,
and, by~\eqref{eq:27092018a2}, for any $y\in I^\circ$,
we have $a^{Eu}_h(y)\sameorder\sqrt h$, $h\to0$.
This is like what we have for the EMCEL scale factors
in Remark~\ref{rem:27092018a1}~(b)
and different from what we have
in Remark~\ref{rem:27092018a1}~(c).\footnote{We
do not compare the EMCEL and the Euler schemes
in the situation of Remark~\ref{rem:27092018a1}~(a)
because the latter falls out of the SDE case,
and hence the Euler scheme is not defined.}
While the EMCEL scheme always converges
(Proposition~\ref{prop:07022020a2}),
the Euler scheme can fail to converge when $\eta$ is irregular.
In Example~5.4 of \cite{aku-jmaa}
it is proved that the Euler scheme does not converge
(even weakly) in the case
$\eta(x)=\frac1{|x|}1_{\bbR\setminus\{0\}}+1_{\{0\}}(x)$,
$x\in I:=\bbR$.
Contrary to the Euler scheme, for the EMCEL scheme,
in the latter case we have $\wh a_h(0)\sameorder h^{1/4}$, $h\to0$
(Remark~\ref{rem:27092018a1}~(c)).\footnote{In this connection, it is worth mentioning that Theorem~2.1 in \cite{yan} states an equivalent condition (and Theorem~2.2 there provides a sufficient condition) in a setting with possibly discontinuous $\eta$ for the Euler scheme to converge weakly to the diffusion $Y$, provided uniqueness in law holds for the SDE (which we have in Example~\ref{ex:sde} due to the Engelbert-Schmidt conditions) and $\eta$ is locally bounded and has at most linear growth. Example~5.4 of \cite{aku-jmaa} does not fall into the setting of \cite{yan} (and, indeed, the Euler scheme fails to converge) because, in that example, $\eta$ is not locally bounded.}

\smallskip
Next let $y\in I^\circ$ and assume that
$\wh a_h(y)\sameorder h^\beta$, $h\to0$,
for some $\beta>0$.
A natural question is then what is the limit of
$\frac{\wh a_h(y)}{h^\beta}$ as $h\to0$
(and if it exists at all).
The claims in Remark~\ref{rem:27092018a1}
do not say anything about this.
We now provide several more precise statements
of this kind (in particular,
improving the claims in Remark~\ref{rem:27092018a1}).

\begin{corollary}\label{cor:06022020a1}
Let $y\in I^\circ$ and $m(\{y\})>0$. Then
$\lim_{h\to0}\frac{\wh a_h(y)}h=\frac2{m(\{y\})}$.
\end{corollary}

\begin{proof}
This result still follows from Lemma~\ref{prop:27092018a1}.
In more detail, the second inequality in~\eqref{eq:04022020a1}
implies
$$
\liminf_{h\to0}\frac{\wh a_h(y)}h\ge\frac2{m(\{y\})}.
$$
Next, the first inequality in~\eqref{eq:04022020a1} yields
\begin{align*}
\limsup_{h\to0}\frac{\wh a_h(y)}h
&\le
\limsup_{h\to0}
\frac2{\sup_{\lambda\in[0,1]}\left\{(1-\lambda)m([y-\lambda\wh a_h(y),y+\lambda\wh a_h(y)])\right\}}\\
&\le\limsup_{h\to0}\frac2{m(\{y\})}=\frac2{m(\{y\})}.
\end{align*}
This completes the proof.
\end{proof}

Below in Propositions \ref{prop:06022020a2} and~\ref{prop:08022020a1}
we discuss the case, where the speed measure $m$ has
the following structure in a neighborhood of some point $y\in I^\circ$:
there is $\alpha\in(-\infty,\frac12)$ and a non-vanishing Borel function $\varphi$ such that
\begin{equation}\label{eq:07022020a3}
m(dx)=\frac2{\varphi^2(x)}|x-y|^{-2\alpha}\,dx
\quad\text{in some neighborhood of }y.
\end{equation}
On the one hand, this allows to improve the claims in
Remark~\ref{rem:27092018a1}~(b) and~(c).
On the other hand, this allows to complement
the above comparison with the Euler scheme
(see Remark~\ref{rem:08022020a1}).

We first need the following notation.
Let $z\in\bbR$, $\eps>0$ and
$f\colon(z-\eps,z+\eps)\setminus\{z\}\to\bbR$
be a real function defined in a deleted neighborhood of $z$. We set
\begin{align*}
|f|^*(z)&=\limsup_{x\to z}|f(x)|,\\
|f|_*(z)&=\liminf_{x\to z}|f(x)|.
\end{align*}

\begin{propo}\label{prop:06022020a2}
Let $y\in I^\circ$.
Assume that there exist $\alpha\in(-\infty,\frac12)$
and a non-vanishing Borel function $\varphi$ such that
the speed measure $m$ has structure~\eqref{eq:07022020a3}.
Then it holds
\begin{equation}\label{eq:09022020a1}
\begin{split}
[(1-2\alpha)(1-\alpha)]^{\frac1{2-2\alpha}}
|\varphi|_*(y)^{\frac1{1-\alpha}}
&\le
\liminf_{h\to0}
\frac{\wh a_h(y)}{h^{\frac1{2-2\alpha}}}\\[1mm]
&\le
\limsup_{h\to0}
\frac{\wh a_h(y)}{h^{\frac1{2-2\alpha}}}
\le
[(1-2\alpha)(1-\alpha)]^{\frac1{2-2\alpha}}
|\varphi|^*(y)^{\frac1{1-\alpha}}.
\end{split}
\end{equation}
\end{propo}

\begin{proof}
We prove only the last inequality in~\eqref{eq:09022020a1}.
The first one is proved in a similar way.
If $|\varphi|^*(y)=\infty$, then there is nothing to prove.
Below we assume that $|\varphi|^*(y)<\infty$.
Choose $h_0\in(0,\ol h]$ such that
$y\in(l_{h_0},r_{h_0})$ and
$(y-\wh a_{h_0}(y),y+\wh a_{h_0}(y))$
is included in the neighborhood, where~\eqref{eq:07022020a3} holds.
Remark~\ref{rem:02022020a1} and Proposition~\ref{prop:07022020a1}
imply that, for all $h\in(0,h_0]$, we have
\begin{equation}\label{eq:09022020a2}
\int_{y-\wh a_h(y)}^{y+\wh a_h(y)}
(\wh a_h(y)-|u-y|)\,|u-y|^{-2\alpha}\frac2{\varphi^2(u)}\,du=2h.
\end{equation}
Consider an arbitrary $\eps>0$.
Then choose $h_1\in(0,h_0]$ such that
\begin{equation}\label{eq:09022020a3}
|\varphi(u)|\le(1+\eps)|\varphi|^*(y)\quad \text{for all }
u\in(y-\wh a_{h_1}(y),y+\wh a_{h_1}(y))
\end{equation}
(this is possible due to Proposition~\ref{prop:09022020a1}).
It follows from
\eqref{eq:09022020a2} and~\eqref{eq:09022020a3} that
$$
\frac2{(1+\eps)^2|\varphi|^*(y)^2}
\int_{y-\wh a_h(y)}^{y+\wh a_h(y)}
(\wh a_h(y)-|u-y|)\,|u-y|^{-2\alpha}\,du\le 2h
\quad 
\text{for all }
h\in(0,h_1].
$$
The integral is explicitly computable, and we get
$$
\frac2{(1+\eps)^2|\varphi|^*(y)^2}\,
\frac{\wh a_h(y)^{2-2\alpha}}{(1-2\alpha)(1-\alpha)}
\le 2h
\quad 
\text{for all }
h\in(0,h_1],
$$
hence
$$
\limsup_{h\to0}
\frac{\wh a_h(y)}{h^{\frac1{2-2\alpha}}}
\le
[(1-2\alpha)(1-\alpha)]^{\frac1{2-2\alpha}}
\left[|\varphi|^*(y)(1+\eps)\right]^{\frac1{1-\alpha}}.
$$
As $\eps>0$ is arbitrary, we obtain the result.
\end{proof}

The next result provides a sufficient condition for the $\liminf$ and $\limsup$ in Proposition~\ref{prop:06022020a2} to coincide.

\begin{propo}\label{prop:08022020a1}
Let $y\in I^\circ$.
Assume that there exist $\alpha\in(-\infty,\frac12)$
and a non-vanishing Borel function $\varphi$ such that
the speed measure $m$ has structure~\eqref{eq:07022020a3}.
Further assume that
\begin{equation}\label{eq:09022020a6}
\lim_{x\nearrow y}|\varphi(x)|=|\varphi|(y-)\in(0,\infty]
\quad\text{and}\quad
\lim_{x\searrow y}|\varphi(x)|=|\varphi|(y+)\in(0,\infty].
\end{equation}
Then
$$
\lim_{h\to0}\frac{\wh a_h(y)}{h^{\frac1{2-2\alpha}}}
=\left[\frac{(1-2\alpha)(2-2\alpha)}{\frac1{|\varphi|^2(y-)}+\frac1{|\varphi|^2(y+)}}\right]^{\frac1{2-2\alpha}}.
$$
\end{propo}

\begin{proof}
A formal proof is obtained along the lines of the proof
of Proposition~\ref{prop:06022020a2}.
We only show some technical steps that need to be elaborated differently.
For notational convenience we set $c_1=|\varphi|(y-)$ and $c_2=|\varphi|(y+)$.
Consider an arbitrary $\eps\in(0,1)$.
By~\eqref{eq:09022020a6}, for sufficiently small $a>0$, we have
$$
\int_y^{y+a}
(a-|u-y|)\,|u-y|^{-2\alpha}\frac2{\varphi^2(u)}\,du
\in[B_{low},B_{up}],
$$
where
\begin{align*}
B_{low}&=\frac{a^{2-2\alpha}}{(1-2\alpha)(2-2\alpha)}
\,\frac2{c_2^2}\,(1-\eps),\\[1mm]
B_{up}&=\frac{a^{2-2\alpha}}{(1-2\alpha)(2-2\alpha)}
\left(\frac2{c_2^2}+\eps\right).
\end{align*}
It is worth noting that $\eps$ appears in the lower and upper bounds
in a non-symmetric way because we need nonnegative bounds
and we need to include the possibility $c_2=\infty$.
With similar bounds for the integral from $y-a$ to $y$, we obtain,
for sufficiently small $h>0$,
\begin{equation}\label{eq:09022020a4}
\frac{\wh a_h(y)^{2-2\alpha}}{(1-2\alpha)(2-2\alpha)}
\left[\frac2{c_1^2}+\frac2{c_2^2}\right](1-\eps)
\le2h\le
\frac{\wh a_h(y)^{2-2\alpha}}{(1-2\alpha)(2-2\alpha)}
\left[\frac2{c_1^2}+\frac2{c_2^2}+2\eps\right].
\end{equation}
On the one hand, \eqref{eq:09022020a4} yields
\begin{equation}\label{eq:09022020a5}
\liminf_{h\to0}
\frac{\wh a_h(y)}{h^{\frac1{2-2\alpha}}}
\ge
\left[
\frac{(1-2\alpha)(2-2\alpha)}{\frac1{c_1^2}+\frac1{c_2^2}+\eps}
\right]^{\frac1{2-2\alpha}},
\end{equation}
which already implies the result in the case $c_1=c_2=\infty$,
as $\eps\in(0,1)$ is arbitrary.
In the case $\min\{c_1,c_2\}<\infty$,
we also get from~\eqref{eq:09022020a4}
$$
\limsup_{h\to0}
\frac{\wh a_h(y)}{h^{\frac1{2-2\alpha}}}
\le
\left[\frac{(1-2\alpha)(2-2\alpha)}{\frac1{c_1^2}+\frac1{c_2^2}}\right]^{\frac1{2-2\alpha}}
\left(\frac1{1-\eps}\right)^{\frac1{2-2\alpha}},
$$
which, together with~\eqref{eq:09022020a5}, concludes the proof.
\end{proof}

\begin{remark}\label{rem:08022020a1}
Let $y\in I^\circ$.
In the setting of Example~\ref{ex:sde},
Proposition~\ref{prop:06022020a2} implies
\begin{equation}\label{eq:09022020a7}
|\eta|_*(y)
\le\liminf_{h\to0}\frac{\wh a_h(y)}{\sqrt h}
\le\limsup_{h\to0}\frac{\wh a_h(y)}{\sqrt h}
\le|\eta|^*(y)
\end{equation}
(because \eqref{eq:07022020a3} holds with $\eta$ in place of $\varphi$ and $\alpha=0$).
In particular, if $\eta$ is continuous at point $y$, then
$$
\lim_{h\to0}\frac{\wh a_h(y)}{\sqrt h}=|\eta(y)|,
$$
which has a clear interpretation that, for small $h>0$,
the EMCEL scheme is close to the Euler one at points,
where $\eta$ is continuous.

Furthermore, if in the setting of Example~\ref{ex:sde} we have~\eqref{eq:09022020a6}
with $\eta$ in place of $\varphi$,
then Proposition~\ref{prop:08022020a1}
improves~\eqref{eq:09022020a7} by establishing
\begin{equation}\label{eq:09022020a8}
\lim_{h\to0}\frac{\wh a_h(y)}{\sqrt h}=
\sqrt{\frac2{\frac1{|\eta|^2(y-)}+\frac1{|\eta|^2(y+)}}},
\end{equation}
i.e., the limit $\lim_{h\to0}\frac{\wh a_h(y)}{\sqrt h}$ is equal to
the power mean with exponent $-2$ of the left and the right limits of $|\eta|$ at~$y$.

To illustrate this observation consider the SDE $dY_t = \sigma(Y_t) dW_t$ with periodic diffusion coefficient 
$$
\sigma(y) = \begin{cases}
1 & \text{if }y \in \cup_{k \in \IZ} [2k, 2k+1), \\
2 & \text{if }y \in \cup_{k \in \IZ} [2k+1, 2k+2).
\end{cases}
$$
The SDE describes a diffusion in a medium that is periodically compounded with two types of layers.
Equation~\eqref{eq:09022020a8} implies that
$$
\lim_{h\to0}\frac{\wh a_h(y)}{\sqrt h}=
\begin{cases}
1&\text{if }y \in \cup_{k \in \IZ} (2k, 2k+1),\\
\sqrt{8/5}&\text{if }y \in \IZ,\\
2&\text{if }y \in \cup_{k \in \IZ} (2k+1, 2k+2)
\end{cases}
$$
with no need to actually compute the EMCEL scale factors. We remark that the scale factor of $\sqrt{8/5}$ at the boundaries of the layers also appears in the diffusion's homogenization limit: the distribution of the solution $Y^\eps$ of the SDE $dY_t = \sigma\left( \frac{Y_t}{\eps}\right) dW_t$, $Y_0 = 0$, converges, as $\eps \downarrow 0$, to the distribution of a BM scaled by $\sqrt{8/5}$ (see \cite{BLP1978} for an introduction into the homogenization theory for periodic SDEs). 

\end{remark}

\subsection{Dependence on the state}\label{sec:cat2}
Next we discuss properties from category~\eqref{it:31012020a2} of the introduction.
Specifically, here we examine how the EMCEL scale factors depend on the state variable, i.e.,
we study the functions $I\ni y\mapsto\wh a_h(y)$,
and observe that the results translate into some good stability features of the scheme.

\begin{theo}[Comparison principle]\label{prop:comp_princ}
For every $h\in(0,\ol h]$ and $z\in \{-1,1\}$,
the mapping $y\mapsto y+\wh a_h(y)z$ is nondecreasing on~$I$.
\end{theo}

Let us discuss the meaning of this result.
If at some time $kh$ the EMCEL$(h)$ approximation
is in position $y\in I$,
then it will be either in $y+\wh a_h(y)$
or in $y-\wh a_h(y)$ at time $(k+1)h$.
Consider two points $y_1<y_2$ in $I$.
Theorem~\ref{prop:comp_princ} suggests to compare
two situations, where at time $kh$ we are in $y_1$
(the 1st situation) or in $y_2$ (the 2nd one),
and asserts that, if we use the same realized
$\xi_{k+1}$ in both situations (recall~\eqref{eq:def_X}),
then, in the 1st situation, we end up in the smaller position
at time $(k+1)h$ than in the 2nd situation.
Notice that this property need not hold for the Euler scheme\footnote{When we speak about the Euler scheme, we restrict ourselves to the SDE case.},
which may result in a kind of ``diverging oscillations''
in the Euler scheme.

We quote Figure~\ref{fig cosh} from \cite{aku-jmaa}\footnote{We
remark that Theorem~\ref{prop:comp_princ},
which holds for \emph{all} possible speed measures $m$,
is a much stronger result than the comparison principle
in \cite{aku-jmaa} and that the proof in \cite{aku-jmaa}
is based on the implicit function theorem,
which cannot be used for all \emph{all} possible
speed measures $m$,
as the involved functions, in particular, $\wh a_h$, are,
in general, not in $C^1$
(the latter claim follows from Theorem~\ref{theo:ode} below).
Thus, the the proof of Theorem~\ref{prop:comp_princ}
uses ideas that are not present in \cite{aku-jmaa}.}
as an example of what can happen
when the SDE coefficients are of superlinear growth
(also see Theorem~2.1 in \cite{HJK}
for a related quantitative statement regarding the Euler scheme).
More precisely, Figure~\ref{fig cosh} considers the numerical performance of the weak Euler scheme\footnote{See the text preceding Corollary~\ref{cor:06022020a1} for the definition of the weak Euler scheme.} and the EMCEL scheme for the SDE $dY_t=\cosh(Y_t)\,dW_t$. To understand why for fixed time step $h>0$ the weak Euler scheme $X^{Eu,h}$ exhibits ``diverging oscillations'' in this example assume that at some time $kh$ the scheme has reached a point\footnote{The same reasoning applies to the case $y<0$ with straightforward modifications.} $y>0$ large enough so that $2y/\cosh(y)<\sqrt{h}$. If we next have an upward jump (i.e., $\xi_{k+1}=1$) then clearly $X^{Eu,h}_{(k+1)h}>y$. In the other case where we next have a downward jump (i.e., $\xi_{k+1}=-1$) the condition $2y/\cosh(y)<\sqrt{h}$ entails that $X^{Eu,h}_{(k+1)h}<-y$. So with probability one we have that the absolute value of the Euler scheme at time $(k+1)h$ is bigger than its absolute value at time $kh$. In fact, the exponential growth of $\cosh$ entails that almost surely $\{k,k+1,\ldots\}\ni n \mapsto |X^{Eu,h}_{nh}|\in \R$ increases super-exponentially, which ultimately leads to the ``diverging oscillations''. Let us next justify why Theorem~\ref{prop:comp_princ} ensures that such explosions cannot happen for the EMCEL scheme. To this end suppose again that at some time $kh$ the scheme has reached a high value $y>0$. If there is a streak of further subsequent upward jumps, then Theorem~\ref{prop:comp_princ} ensures that the size of each jump $\wh a_h(\wh X_{lh})$, $l\ge k$, is at most linear in $\wh X_{lh}$ (cf.\ also Corollary~\ref{prop:26092018a1} below). Eventually, at some time $nh$ there will be a downward jump. In this case Theorem~\ref{prop:comp_princ} ensures that $\wh X_{(n+1)h} =\sup_{y\le \wh X_{nh} }(y-\wh a_h(y)) \ge -\wh a_h(0)$ so that the EMCEL scheme opposed to the Euler scheme does not overjump a large region around $0$ but rather jumps back into a ``stable'' region around $0$ (in the setting of the bottom figure of Figure~\ref{fig cosh} a downward jump from any level $y\ge 3$ leads to a value close to $1$ after the jump).

\begin{figure}[!htb]
\centering
\includegraphics[width=0.7\textwidth]{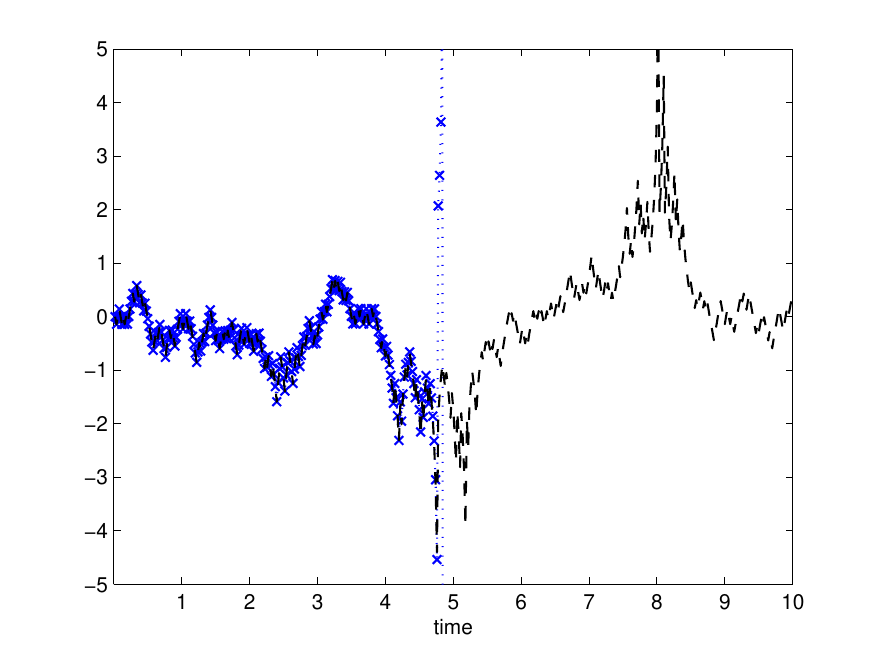}\\
\includegraphics[width=0.7\textwidth]{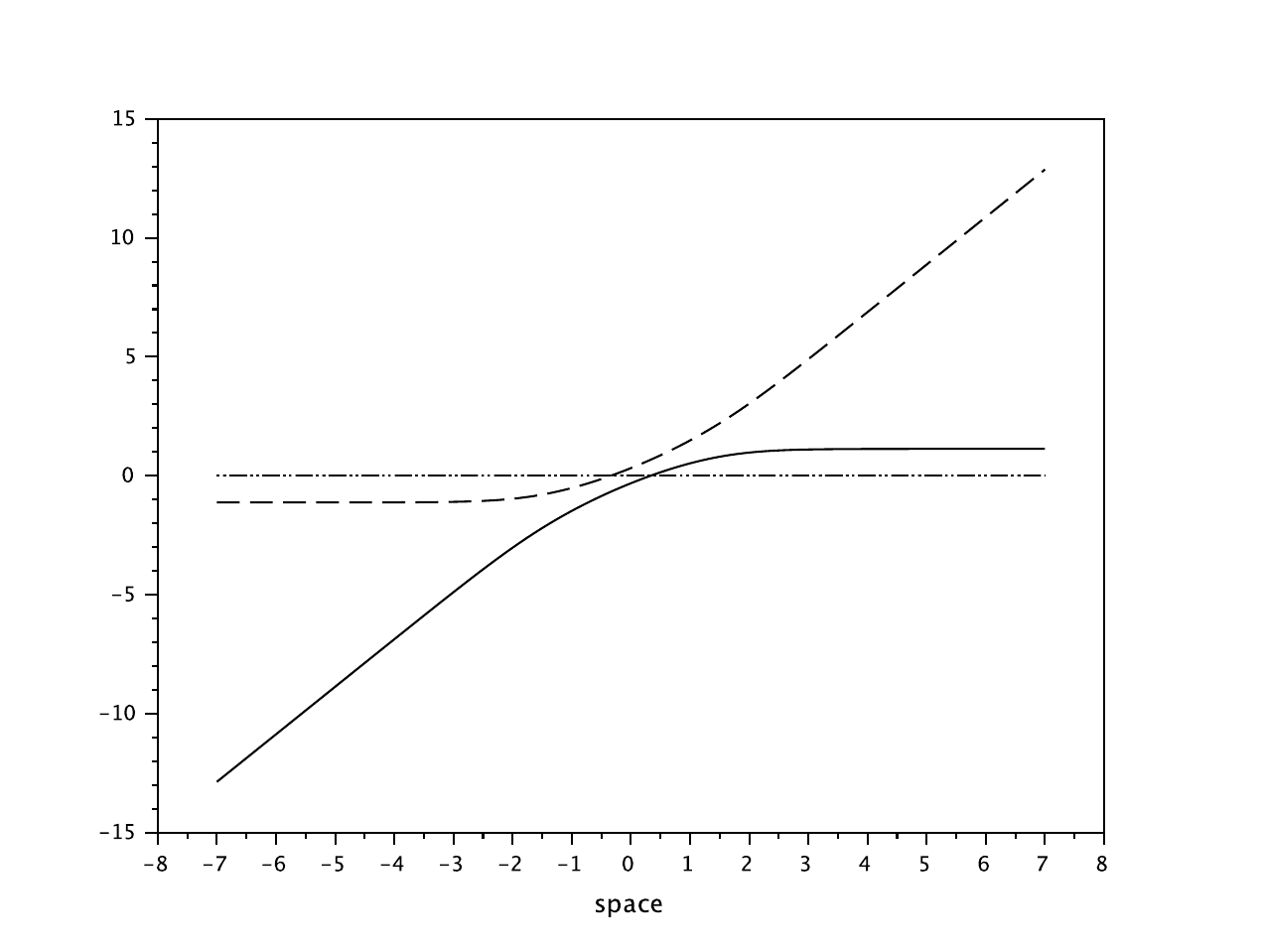}

{\caption{\footnotesize  The figure on top shows two realizations of discrete approximations to the SDE $dY_t = \cosh(Y_t)\,dW_t$ with $Y_0 = 0$. The dashed line depicts the realization based on the EMCEL scheme. The crosses show the realization obtained with the Euler scheme. Both use the same realized increments $(\xi_k)$. Notice that the approximations are nearly identical until shortly before time~5. The large absolute values entail that the Euler approximation explodes and eventually aborts, whereas the dashed approximation easily continues.
In the bottom figure the solid and dashed lines
are the graphs of the functions
$y\mapsto y-\wh a_h(y)$
and
$y\mapsto y+\wh a_h(y)$.
The dash-dotted line indicates level zero.
As outlined above the monotonicity 
of both functions implies that divergent oscillations
are impossible in the EMCEL scheme.
}\label{fig cosh}}
\end{figure}

It is instructive to discuss relations and differences between the comparison principle of Theorem~\ref{prop:comp_princ} and comparison theorems for solutions of SDEs
(see, e.g., Theorem~1.4 in \cite{LeGall1984} for It\^o SDEs and Theorem~4.2 in \cite{BassChen2001} for Stratonovich ones):
\begin{itemize}
\item
Comparison theorems for SDEs are pathwise results, and they apply in the situations when the SDE has the pathwise uniqueness property.
And this makes perfect sense, as if an SDE that has a solution does not satisfy pathwise uniqueness, then one can find two different solutions to it with the same driving Brownian motion.
That is, there are no such pathwise comparison results beyond the case of pathwise uniqueness.
It is worth noting that pathwise uniqueness can fail even in the situation of Example~\ref{ex:sde} even with a continuous~$\eta$ (see \cite{Barlow1982}).
\item
The comparison principle of Theorem~\ref{prop:comp_princ} is a pathwise property of the EMCEL scheme only,
which, as discussed above, translates into a good stability feature of the scheme.
However, it does not imply any pathwise comparison result in the spirit of Theorem~1.4 in \cite{LeGall1984},
as the EMCEL scheme approximates general diffusions only in the weak sense.
On the other hand, Theorem~\ref{prop:comp_princ} applies to the EMCEL approximation of \emph{every} general diffusion $Y$
and thus goes far beyond the SDE case under pathwise uniqueness
(e.g., $Y$ can be a solution to an SDE from \cite{Barlow1982}, $Y$ can have sticky features, etc.).
\end{itemize}

\begin{proof}[Proof of Theorem~\protect\ref{prop:comp_princ}]
For $y\in I^\circ$ and $a>0$ such that $y\pm a\in I$,
we use the notation
\begin{equation}\label{eq:26092018a1}
G(y,a)=\int_{(y-a,y+a)} (a-|u-y|)\,m(du).
\end{equation}
We fix $h\in(0,\ol h]$ and elements $y_1<y_2$ of $I$.
Define $a_i=\wh a_h(y_i)$, $i=1,2$.
We need to show that
\begin{equation}\label{eq:26092018a3}
y_1+a_1\le y_2+a_2
\quad\text{and}\quad
y_1-a_1\le y_2-a_2.
\end{equation}
This is clear whenever $y_1\in I\setminus I^\circ$
or $y_2\in I\setminus I^\circ$
(recall that, by construction,
$\wh a_h(l)=\wh a_h(r)=0$ and, for all $y\in I^\circ$,
we have $y\pm\wh a_h(y)\in[l,r]$, see~\eqref{eq:08112019a3}).
Below we, therefore,
assume $y_1,y_2\in I^\circ$ and consider four cases.

\smallskip
1. Let both endpoints $l$ and $r$ be inaccessible.
Then we have
\begin{equation}\label{eq:26092018a2}
G(y_1,a_1)=2h=G(y_2,a_2).
\end{equation}
The expression for $G(y,a)$ on the right-hand side
of~\eqref{eq:26092018a1}
together with~\eqref{eq:26092018a2} imply
that neither of the intervals
$(y_1-a_1,y_1+a_1)$
and
$(y_2-a_2,y_2+a_2)$
contains the other.
This yields~\eqref{eq:26092018a3}.

\smallskip
2. Let $l$ be accessible and $r$ inaccessible.
We first show the second statement in~\eqref{eq:26092018a3}.
If $y_1-a_1=l$, then the statement is clear,
as, by construction, for all $y\in I^\circ$,
it holds that $y-\wh a_h(y)\ge l$.
If $y_1-a_1>l$, then we have
\begin{equation}\label{eq:26092018a4}
G(y_1,a_1)=2h\ge G(y_2,a_2).
\end{equation}
If we now assume that
$y_1-a_1>y_2-a_2$,
then the interval
$(y_2-a_2,y_2+a_2)$
strictly contains the interval
$(y_1-a_1,y_1+a_1)$.
Notice that the integrand in~\eqref{eq:26092018a1} is strictly positive,
the integrand corresponding to the bigger interval $(y_2-a_2,y_2+a_2)$
dominates the one corresponding to the smaller interval $(y_1-a_1,y_1+a_1)$ on that interval
and that $m$ has full support by~\eqref{eq:06072018a1}.
This 
implies $G(y_2,a_2)>G(y_1,a_1)$ and hence contradicts~\eqref{eq:26092018a4}.

Next we show the first statement in~\eqref{eq:26092018a3}.
In the case $y_2-a_2=l$, the statement follows from the fact that
$y_1-a_1\ge l$ and hence $a_1<a_2$.
If $y_2-a_2>l$, then we have
\begin{equation}\label{eq:26092018a5}
G(y_1,a_1)\le2h=G(y_2,a_2).
\end{equation}
Assume that
$y_1+a_1>y_2+a_2$.
Then the interval
$(y_1-a_1,y_1+a_1)$
strictly contains the interval
$(y_2-a_2,y_2+a_2)$,
which, together with~\eqref{eq:26092018a1},
implies
$G(y_1,a_1)>G(y_2,a_2)$
and hence contradicts~\eqref{eq:26092018a5}.

\smallskip
3. The case, where $l$ is inaccessible and $r$ is accessible,
is symmetric to case 2.

\smallskip
4. Let finally both $l$ and $r$ be accessible.
We prove only the first statement in~\eqref{eq:26092018a3},
as the second one is symmetric.
If $y_2+a_2=r$, then the statement follows from
the fact that $y_1+a_1\le r$.
In the case $y_2-a_2=l$, the statement follows from the fact that
$y_1-a_1\ge l$ and hence $a_1<a_2$.
In the remaining case
$y_2\pm a_2\in I^\circ$,
we have
$$
G(y_1,a_1)\le2h=G(y_2,a_2),
$$
and the argument after~\eqref{eq:26092018a5}
yields the desired statement.
\end{proof}

\begin{remark}\label{rem:07042020a1}
Analyzing the proof of Theorem~\ref{prop:comp_princ} in more detail,
we obtain that the following more precise version of the comparison principle holds true:

Let $h\in(0,\ol h]$. Then the mapping
$$
y\mapsto y+\wh a_h(y)\text{ is strictly increasing on }(l,r_h)\text{ and constant on }[r_h,r),
$$
and the mapping
$$
y\mapsto y-\wh a_h(y)\text{ is constant on }(l,l_h]\text{ and strictly increasing on }(l_h,r).
$$
\end{remark}

\begin{corollary}[Smoothing and tempered growth behavior]\label{prop:26092018a1}
For any $h\in(0,\ol h]$, the function $I\ni y\mapsto\wh a_h(y)$
is Lipschitz continuous on $I$ with Lipschitz constant $1$, i.e.,
\begin{equation}
|\wh a_h(y_1)-\wh a_h(y_2)|\le|y_1-y_2|
\quad\text{for all }y_1,y_2\in I\text{ and }h\in(0,\ol h].
\end{equation}
Moreover, there exists a constant $C_0\in[0,\infty)$ such that
\begin{equation}\label{eq:09022020b1}
\wh a_h(y)\le C_0+|y|
\quad\text{for all }y\in I\text{ and }h\in(0,\ol h].
\end{equation}
\end{corollary}

\begin{proof}
The first statement is an immediate consequence of the comparison principle.
The second statement easily follows from the first one with, e.g.,
$C_0=|y_0|+\sup_{h\in(0,\ol h]}\wh a_h(y_0)=|y_0|+\wh a_{\ol h}(y_0)<\infty$
(recall Proposition~\ref{prop:07022020a1}),
where $y_0$ is an arbitrary point in~$I$.
\end{proof}

\begin{remark}
Corollary~\ref{prop:26092018a1} is named as above
to stress the difference with the Euler scheme in the SDE case,
where the Euler scale factors $a^{Eu}_h(y)=\eta(y)\sqrt h$,
as functions of $y$,
inherit irregularities and the growth from $\eta$.
On the contrary, the EMCEL scale factors $\wh a_h$
are as described in Corollary~\ref{prop:26092018a1},
no matter how irregular $\eta$ is in the SDE case
and also beyond the SDE case.
\end{remark}

Corollary~\ref{prop:26092018a1}
provides the functional bound $C_0+|y|$
(independent of $h$)
for all functions $\wh a_h$, $h\in(0,\ol h]$.
We also know that $\wh a_h(y)\to0$, $h\to0$, for all fixed $y\in I$
(recall Proposition~\ref{prop:09022020a1}
and $\wh a_h(l)=\wh a_h(r)=0$).
A natural question is now to find a functional bound
for $\wh a_h$, $h\in(0,\ol h]$,
that depends on $h$ and vanishes as $h\to0$.
However, this does not seem to be feasible in general,
as the order of convergence (in $h$)
of $\wh a_h(y)$ to zero can be different in different points $y$.
The discussion following Remark~\ref{rem:27092018a1}
suggests that the precise forms of such functional bounds
have to depend on the structure of the speed measure $m$.
We, finally, present a result of such kind.

\begin{propo}\label{prop:09022020b1}
Suppose that we have
\begin{equation}\label{eq:09022020b2}
m(dx)\ge\frac{dx}{g(|x|)}\quad\text{on }I^\circ
\end{equation}
(understood in the integral form) with some positive nondecreasing function
$g\colon[0,\infty)\to(0,\infty)$.
Then, with any constant $C_0\in[0,\infty)$
satisfying~\eqref{eq:09022020b1}, we obtain
\begin{equation}\label{eq:lin_growth_sf}
\wh a_h(y)\le\sqrt{2g(C_0+2|y|)h}
\quad\text{for all }y\in I\text{ and }h\in(0,\ol h].
\end{equation}
\end{propo}

We list a couple of specific functional bounds for $\wh a_h$
implied by Proposition~\ref{prop:09022020b1}.

\smallskip
(a) Let \eqref{eq:09022020b2} be satisfied with
$g(x)=c(1+x^p)$, $x\in[0,\infty)$, for some $c,p\in(0,\infty)$.
Then there exists $C\in(0,\infty)$ such that
$$
\wh a_h(y)\le C(1+|y|^{p/2})\sqrt h
\quad\text{for all }y\in I\text{ and }h\in(0,\ol h].
$$

\smallskip
(b) Let \eqref{eq:09022020b2} be satisfied with
$g(x)=c\exp\{px\}$, $x\in[0,\infty)$, for some $c,p\in(0,\infty)$.
Then there exists $C\in(0,\infty)$ such that
$$
\wh a_h(y)\le C\exp\{p|y|\}\sqrt h
\quad\text{for all }y\in I\text{ and }h\in(0,\ol h].
$$

\begin{proof}[Proof of Proposition~\protect\ref{prop:09022020b1}]
It follows from~\eqref{eq:08112019a1} and~\eqref{eq:09022020b2}
that, for all $y\in I^\circ$ and $h\in(0,\ol h]$, it holds
\begin{equation*}
\begin{split}
2h&\ge\int_{(y-\wh a_h(y),y+\wh a_h(y))}
(\wh a_h(y)-|u-y|)\,m(du)
\\[1mm]
&\ge\int_{y-\wh a_h(y)}^{y+\wh a_h(y)}
\frac{\wh a_h(y)-|u-y|}{g(|u|)}\,du
=\wh a_h(y)^2
\int_{-1}^1
\frac{1-|z|}{g(|y+z\wh a_h(y)|)}\,dz
\\[1mm]
&\ge
\frac{\wh a_h(y)^2}{\sup_{z\in[-1,1]}g(|y+z\wh a_h(y)|)}
=\frac{\wh a_h(y)^2}{g(|y|+|\wh a_h(y)|)}.
\end{split}
\end{equation*}
The claim now follows from~\eqref{eq:09022020b1}.
\end{proof}

\subsection{ODE for the scale factors}\label{sec:cat3}
We, finally, turn to properties from category~\eqref{it:31012020a3} of the introduction,
i.e., the properties that help implementing the EMCEL scheme in specific situations.
It follows from the discussion in Remark~\ref{rem:02022020a1}
that the main challenge in implementing the EMCEL scheme
is to determine the scale factor $\wh a_h$ (for a fixed $h\in(0,\ol h]$)
inside $(l_h,r_h)$ because it requires to solve the nonlinear equation~\eqref{eq:08112019a2}
for all $y\in(l_h,r_h)$.
On the contrary, there is no problem to determine $\wh a_h$ outside $(l_h,r_h)$
(recall~\eqref{eq:22022019a4}).

Theorem~\ref{theo:ode} below shows that $\wh a_h$ is a unique solution to ODE~\eqref{eq:28042021a1} inside $(l_h,r_h)$.
Thus, in order to implement the scheme, it is enough to solve~\eqref{eq:08112019a2} numerically only for some $y_0\in(l_h,r_h)$
(not for all $y\in(l_h,r_h)$), which provides the initial condition for the ODE, and then to apply an appropriate ODE solver.
It is worth noting that the ODE itself does not depend on the discretization parameter $h$.
Dependence on $h$ comes into the picture through the initial condition (solving~\eqref{eq:08112019a2} for some $y_0\in(l_h,r_h)$).

In fact, a solution to ODE~\eqref{eq:28042021a1} is understood in the sense that it is an absolutely continuous function satisfying~\eqref{eq:28042021a1} almost everywhere
(with respect to the Lebesgue measure).
In general, we cannot require~\eqref{eq:28042021a1} everywhere, as we treat \emph{all} possible speed measures (they can, e.g., have atoms).
Essentially, the first part of Theorem~\ref{theo:ode} deals with existence and the second with uniqueness for ODE~\eqref{eq:28042021a1}.
The minimal requirement for the uniqueness is exactly the one mentioned above:
the solution must be absolutely continuous and satisfy~\eqref{eq:28042021a1} almost everywhere.
But in the existence part of Theorem~\ref{theo:ode} we provide more detail about what \emph{de facto} holds for the EMCEL scale factor $\wh a_h$.
For instance, it turns out that $\wh a_h$ is differentiable everywhere on $(l_h,r_h)$ except at most on a
countable set, regardless of how ``irregular'' the speed measure $m$ is.

Below we use the notation $\mu_L$ for the Lebesgue measure.

\begin{theo}\label{theo:ode}
Let $h\in (0,\ol h]$.

(i) For all $y\in (l_h,r_h)$ the right derivative $\partial_+\wh a_h(y)=\lim_{\eps \searrow 0}\frac{\wh a_h(y+\eps)-\wh a_h(y)}{\eps}$ and the left derivative $\partial_-\wh a_h(y)=\lim_{\eps \searrow 0}\frac{\wh a_h(y)-\wh a_h(y-\eps)}{\eps}$ exist and it holds that
\begin{equation}\label{eq:right_der}
\partial_+\wh a_h(y)=\frac{m((y-\wh a_h(y),y])-m((y,y+\wh a_h(y)])}{m((y-\wh a_h(y),y+\wh a_h(y)])}
\end{equation}
and
\begin{equation}\label{eq:left_der}
\partial_-\wh a_h(y)=\frac{m([y-\wh a_h(y),y))-m([y,y+\wh a_h(y)))}{m([y-\wh a_h(y),y+\wh a_h(y)))}.
\end{equation}
Moreover, $\partial_+\wh a_h$ is c\`{a}dl\`{a}g, $\partial_-\wh a_h$ is c\`{a}gl\`{a}d and, for all $y\in (l_h,r_h)$, we have
$\partial_+\wh a_h(y-):=\lim_{\eps\searrow 0}\partial_+\wh a_h(y-\eps)=
\partial_-\wh a_h(y)$. In particular, there exists an at most countable set $\mathcal N\subset (l_h,r_h)$ such that $\wh a_h'$ exists on $(l_h,r_h)\setminus \mathcal N$ and for all $y\in (l_h,r_h)\setminus \mathcal N$ we have $\wh a_h'(y)=\partial_+\wh a_h(y)=\partial_-\wh a_h(y)$.

\smallskip (ii)
Let $y_0\in (l_h,r_h)$ and let $a\colon (l_h,r_h)\to (0,\infty)$ be an absolutely continuous function
on compact subintervals of $(l_h,r_h)$ that satisfies, 
for $\mu_L$-almost all $y\in (l_h,r_h)$,
\begin{equation}\label{eq:28042021a1}
a'(y)=F(y,a(y))
\end{equation}
and $a(y_0)=\wh a_h(y_0)$. Here, the function $F\colon(l_h,r_h)\times(0,\infty)\to\bbR$ is defined by the formulas
\begin{align}
F(y,a)&=\frac{m((y-a,y])-m((y,y+a])}{m((y-a,y+ a])},\quad(y,a)\in\cG,
\label{eq:28042021a2}\\
F(y,a)&=0,\quad(y,a)\in(l_h,r_h)\times(0,\infty)\setminus\cG,
\label{eq:28042021a3}
\end{align}
where the domain $\cG$ is defined as
\begin{equation}\label{eq:28042021a4}
\cG=\{(y,a)\in(l_h,r_h)\times(0,\infty)\colon y\pm a\in I^\circ\}.
\end{equation}
Then we have $a(y)=\wh a_h(y)$ for all $y\in (l_h,r_h)$.
\end{theo}

To explain ODE~\eqref{eq:28042021a1} in more detail we now make several comments:
\begin{itemize}
\item
For $y\in(l_h,r_h)$, we have $y\pm\wh a_h(y)\in I^\circ$, i.e., $(y,\wh a_h(y))\in\cG$. In other words, for the EMCEL scale factor, the right-hand side of~\eqref{eq:28042021a1} reduces to the right-hand side of~\eqref{eq:28042021a2} only, i.e., \eqref{eq:28042021a3}~is not needed.
\item
In~\eqref{eq:28042021a3} we extend the function $F$ beyond the domain $\cG$ only because we need a well-defined right-hand side of~\eqref{eq:28042021a1} for functions $a$ that \emph{a priori} need not coincide with the EMCEL scale factor.\footnote{Notice that the right-hand side of~\eqref{eq:28042021a2} can fail to be well-defined outside $\cG$ because the speed measure $m$ can be infinite near the boundary points of $I$; cf.~\eqref{eq:06072018a1}.}
\item
On the other hand, it will become clear from the proof that it does not matter how to define $F$ outside $\cG$. In this sense, \eqref{eq:28042021a3}~is not important.
\end{itemize}

The proof of Theorem~\ref{theo:ode} is based on the following lemma.

\begin{lemma}\label{lem:help_ode}
Fix $h\in (0,\ol h]$.
Let $a\colon (l_h,r_h)\to (0,\infty)$ be a function such that,
for all $y_1,y_2\in (l_h,r_h)$, it holds $y_1\pm a(y_1)\in I^\circ$ and $|a(y_2)-a(y_1)|\le|y_2-y_1|$. 
Define the function $H\colon  (l_h,r_h)\to (0,\infty)$ by the formula
$H(y)=\int_{(y-a(y),y+a(y))}(a(y)-|u-y|)\, m(du)$.
Then $H$ is locally Lipschitz continuous and, for all $y_1<y_2$ in $(l_h,r_h)$ sufficiently close to each other\footnote{In the sense
$y_2-a(y_2)\le y_1<y_2\le y_1+a(y_1)$,
which is needed for the right ordering between the endpoints of the intervals involved in~\eqref{eq:help_ode}.}, we have the representation
\begin{equation}\label{eq:help_ode}
\begin{split}
H(y_2)-H(y_1)&=(a(y_2)-a(y_1))m([y_2-a(y_2),y_1+a(y_1)])\\
&\quad+(y_2-y_1)(m([y_2,y_1+a(y_1)])-m([y_2-a(y_2),y_1]))+R(y_1,y_2),
\end{split}
\end{equation}
where, for the remainder term $R(y_1,y_2)$ in~\eqref{eq:help_ode}, it holds
\begin{align}
&R(y,y+\Delta y)\in o(\Delta y),\quad\Delta y\searrow0,
\label{eq:07042020a1}\\
&R(y-\Delta y,y)\in o(\Delta y),\quad\Delta y\searrow0,
\label{eq:07042020a2}
\end{align}
for all $y\in(l_h,r_h)$.
\end{lemma}

\begin{proof}[Proof of Theorem~\protect\ref{theo:ode}]
(i) First note that it follows from Corollary~\ref{prop:26092018a1} that $\wh a_h$ is Lipschitz continuous on $(l_h,r_h)$ with Lipschitz constant $1$. Therefore we are in a position to apply Lemma~\ref{lem:help_ode} with $a=\wh a_h$ in the notation of Lemma~\ref{lem:help_ode} and we denote by $\wh H$ the associated function $H$. It follows from Remark~\ref{rem:02022020a1} that for all $y\in (l_h,r_h)$ we have 
$\wh H(y)=2h$. Therefore, we obtain from \eqref{eq:help_ode} and~\eqref{eq:07042020a1} that,
for all $y\in (l_h,r_h)$ and $\Delta y>0$,
\begin{equation}
\begin{split}
\frac{\wh a_h(y+\Delta y)-\wh a_h(y)}{\Delta y}
&=\frac{m([y+\Delta y-\wh a_h(y+\Delta y),y])-m([y+\Delta y,y+\wh a_h(y)])}{m([y+\Delta y-\wh a_h(y+\Delta y),y+\wh a_h(y)])}\\[1mm]
&\quad+o(1),\quad\Delta y\searrow0.
\end{split}
\end{equation}
This yields~\eqref{eq:right_der} (recall that the functions $y\mapsto y\pm \wh a_h(y)$ are strictly increasing on $(l_h,r_h)$ by Remark~\ref{rem:07042020a1}).
In a similar way, using~\eqref{eq:07042020a2} instead of~\eqref{eq:07042020a1}, we obtain~\eqref{eq:left_der}.
The further claims in (i) follow from these two formulas.

\smallskip (ii) We first prove the claim under the additional assumption $y\pm a(y)\in I^\circ$ for all $y\in(l_h,r_h)$
(notice that the EMCEL scale factor $\wh a_h$ satisfies this assumption).
It follows from~\eqref{eq:28042021a1} that for $\mu_L$-almost all $y\in (l_h,r_h)$ we have $|a'(y)|<1$ (recall~\eqref{eq:06072018a1}).
As $a$ is absolutely continuous, for $y_1<y_2$ in $(l_h,r_h)$, we have
\begin{equation}\label{eq:08042020a1}
|a(y_2)-a(y_1)|=\left|\int_{y_1}^{y_2}a'(y)\,dy\right|\le\int_{y_1}^{y_2}|a'(y)|\,dy<y_2-y_1
\end{equation}
and, in particular, the functions $y\mapsto y\pm a(y)$ are strictly increasing on $(l_h,r_h)$.
Notice that the latter implies that the set
\begin{equation}\label{eq:count_set}
\left\{y\in(l_h,r_h):m\left(\{y-a(y)\}\cup\{y\}\cup\{y+a(y)\}\right)>0\right\}
\end{equation}
is at most countable, hence $\mu_L$-negligible.
Due to~\eqref{eq:08042020a1} we can apply Lemma~\ref{lem:help_ode}.
We obtain that the function $H$, as defined in Lemma~\ref{lem:help_ode}, is locally Lipschitz continuous
and hence absolutely continuous on every compact subinterval of $(l_h,r_h)$. Moreover,
it follows from \eqref{eq:help_ode}, \eqref{eq:07042020a1}, \eqref{eq:07042020a2}
and the fact that the set in \eqref{eq:count_set} is $\mu_L$-negligible
that, for $\mu_L$-almost all $y\in (l_h,r_h)$, we have
$$
H'(y)=a'(y)m((y-a(y),y+ a(y)])+m((y,y+ a(y)])-m((y-a(y),y])=0.
$$
Consequently, $H$ is constant on $(l_h,r_h)$ and it follows for all $y\in (l_h,r_h)$ that
$$
\int_{(y-a(y),y+a(y))}(a(y)-|u-y|)\, m(du)=H(y)=H(y_0)=2h.
$$
The claim that $a$ and $\wh a_h$ are identical now follows from Remark~\ref{rem:02022020a1}.

It remains to drop the assumption $y\pm a(y)\in I^\circ$ for all $y\in(l_h,r_h)$.
As $a(y_0)=\wh a_h(y_0)$, it holds $y_0\pm a(y_0)\in I^\circ$.
By the continuity of $a$, we get $y\pm a(y)\in I^\circ$ in a sufficiently small neighborhood of $y_0$.
Hence, by the considerations above, $a$ and $\wh a_h$ coincide in this neighborhood of $y_0$.
Define
\begin{align*}
\wt l&=\inf\{y\in(l_h,y_0]:a(z)=\wh a_h(z)\text{ for all }z\in[y,y_0]\}\quad(\in[l_h,y_0)),\\
\wt r&=\sup\{y\in[y_0,r_h):a(z)=\wh a_h(z)\text{ for all }z\in[y_0,y]\}\quad(\in(y_0,r_h]).
\end{align*}
If $\wt l>l_h$, we repeat the preceding argumentation with $y_0$ replaced by $\wt l$
and conclude that $a$ and $\wh a_h$ coincide in some neighborhood of $\wt l$,
which contradicts the definition of $\wt l$.
Thus, $\wt l=l_h$. Similarly, $\wt r=r_h$.
This completes the proof.
\end{proof}

\begin{proof}[Proof of Lemma~\protect\ref{lem:help_ode}]
Throughout the proof we work with various choices of $y_1<y_2$ in $(l_h,r_h)$.
Note that if $y_2-y_1$ is small enough we have that $y_2-a(y_2)\le y_1<y_2\le y_1+a(y_1)$. Moreover, it follows from the assumption that $a$ is Lipschitz continuous on $(l_h,r_h)$ with Lipschitz constant $1$ that $y_1-a(y_1)\le y_2-a(y_2)$ and 
$y_1+a(y_1)\le y_2+a(y_2)$. To summarize, we have for $y_2-y_1$ small enough that
$$y_1-a(y_1)\le y_2-a(y_2)\le y_1 < y_2 \le y_1+a(y_1)\le y_2+a(y_2).$$
Therefore, we have that
\begin{equation}\label{eq:help_ode_1}
\begin{split}
H(y_2)-H(y_1)&=
-\int_{(y_1-a(y_1),y_2-a(y_2))}(a(y_1)-|u-y_1|)\,m(du)\\
&\quad +
\int_{[y_2-a(y_2),y_1+a(y_1)]}(a(y_2)-|u-y_2|-a(y_1)+|u-y_1|)\,m(du)\\
&\quad +
\int_{(y_1+a(y_1),y_2+a(y_2))}(a(y_2)-|u-y_2|)\,m(du)\\
&=
-\int_{(y_1-a(y_1),y_2-a(y_2))}(a(y_1)+u-y_1)\,m(du)\\
&\quad +
(a(y_2)-a(y_1))m([y_2-a(y_2),y_1+a(y_1)])
\\
& \quad +
\int_{[y_2-a(y_2),y_1+a(y_1)]}(|u-y_1|-|u-y_2|)\,m(du)\\
&\quad +
\int_{(y_1+a(y_1),y_2+a(y_2))}(a(y_2)-u+y_2)\,m(du).
\end{split}
\end{equation}
For each fixed $y_1$ and moving $y_2$ such that $y_2-y_1\searrow 0$ we have
\begin{equation}
\begin{split}\label{eq:help_ode_2}
&\left|\int_{(y_1-a(y_1),y_2-a(y_2))}(a(y_1)+u-y_1)\,m(du)\right |
=\int_{(y_1-a(y_1),y_2-a(y_2))}(a(y_1)+u-y_1)\,m(du)\\
&\le (y_2-y_1+a(y_1)-a(y_2))m((y_1-a(y_1),y_2-a(y_2)))\\
&\le 2(y_2-y_1)m((y_1-a(y_1),y_2-a(y_2)))\in o(y_2-y_1)
\end{split}
\end{equation}
and similarly
\begin{multline}\label{eq:help_ode_3}
\left|\int_{(y_1+a(y_1),y_2+a(y_2))}(a(y_2)-u+y_2)\,m(du)\right |\\
\le 2(y_2-y_1)m((y_1+a(y_1),y_2+a(y_2)))\in o(y_2-y_1).
\end{multline}
Moreover, it holds that
\begin{multline}\label{eq:help_ode_4}
\int_{[y_2-a(y_2),y_1+a(y_1)]}(|u-y_1|-|u-y_2|)\,m(du)\\
=
(y_2-y_1)(m([y_2,y_1+a(y_1)])-m([y_2-a(y_2),y_1]))+
\int_{(y_1,y_2)}(2u-y_1-y_2)\,m(du)
\end{multline}
and that, again for a fixed $y_1$ and moving $y_2$ such that $y_2-y_1\searrow 0$,
\begin{equation}\label{eq:help_ode_5}
\left | \int_{(y_1,y_2)}(2u-y_1-y_2)\,m(du)\right|\le (y_2-y_1)m((y_1,y_2)) \in o(y_2-y_1).
\end{equation}
Combining \eqref{eq:help_ode_1}--\eqref{eq:help_ode_5} we obtain,
for a fixed $y_1$ and moving $y_2$ such that $y_2-y_1\searrow 0$,
\begin{equation}
\begin{split}
H(y_2)-H(y_1)&=(a(y_2)-a(y_1))m([y_2-a(y_2),y_1+a(y_1)])\\
&\quad+(y_2-y_1)(m([y_2,y_1+a(y_1)])-m([y_2-a(y_2),y_1]))+o(y_2-y_1),
\end{split}
\end{equation}
which is \eqref{eq:help_ode} and~\eqref{eq:07042020a1}.
Property~\eqref{eq:07042020a2} follows from similar considerations,
only with fixed $y_2$ and moving $y_1$ such that $y_2-y_1\searrow0$.
Moreover, the preceding calculations imply that, for $y_2-y_1$ small enough,
\begin{equation}
\begin{split}
\left|\frac{H(y_2)-H(y_1)}{y_2-y_1}\right|&\le m([y_2-a(y_2),y_1+a(y_1)]) \left|\frac{a(y_2)-a(y_1)}{y_2-y_1}\right| \\
&\quad+2(m((y_1-a(y_1),y_2-a(y_2)))+m((y_1+a(y_1),y_2+a(y_2))))\\
&\quad+(m([y_2,y_1+a(y_1)])-m([y_2-a(y_2),y_1]))+m((y_1,y_2))\\
&\le 
2m((y_1-a(y_1),y_2+a(y_2))<\infty
\end{split}
\end{equation}
because, by the assumptions, $y_1-a(y_1),y_2+a(y_2)\in I^\circ$ (also recall~\eqref{eq:06072018a1}).
This implies that $H$ is Lipschitz continuous on every compact subinterval of $(l_h,r_h)$. This completes the proof of Lemma~\ref{lem:help_ode}.
\end{proof}

We now mention the following immediate consequence of Theorem~\ref{theo:ode} for the case when the speed measure $m$ does not have atoms in~$I^\circ$.

\begin{corollary}\label{cor:07042020a1}
Assume $m(\{y\})=0$ for any $y\in I^\circ$.
Then, for any $h\in(0,\ol h]$, the restriction of the EMCEL scale factor $\wh a_h$  to $(l_h,r_h)$ is a $C^1$ function that coincides with the unique solution to the initial value problem
\begin{equation}\label{eq:07042020a3}
a'(y)=F(y,a(y)),\quad y\in(l_h,r_h),\quad a(y_0)=\wh a_h(y_0),
\end{equation}
for any fixed $y_0\in(l_h,r_h)$.
\end{corollary}

We stress that existence and uniqueness on $(l_h,r_h)$ for the initial value problem~\eqref{eq:07042020a3} are also claimed in Corollary~\ref{cor:07042020a1}. It is interesting to compare this with what we can get concerning the existence and uniqueness for~\eqref{eq:07042020a3} from standard results on ODEs:
\begin{itemize}
\item
The assumption $m(\{y\})=0$ for any $y\in I^\circ$ of Corollary~\ref{cor:07042020a1} implies that the function $F$ is continuous in the domain $\cG$ (recall \eqref{eq:28042021a2} and~\eqref{eq:28042021a4}). 
Therefore, Peano's existence theorem
(e.g., see Theorem~II.2.1 in \cite{Hartman2002} or Theorem~2.19 in \cite{Teschl2012})
yields that there is a solution to the initial value problem~\eqref{eq:07042020a3}, and it can be extended up to a boundary of $\cG$ (e.g., see Theorem~II.3.1 in \cite{Hartman2002} or \S~11 in \cite{Petrovsky1984}). This is weaker than the existence on the whole $(l_h,r_h)$ because the boundary of $\cG$ can be achieved at points with $y$-coordinates that are strictly between $l_h$ and $r_h$.
\item
Peano's existence theorem does not say anything about uniqueness.\footnote{E.g., the ODE $a'(y)=2\sqrt{|a|}$, $a(0)=0$, has different solutions $a\equiv0$ and $a(y)=y^2\sgn y$ (and there are other ones).
For illuminating examples of ODEs of the form $a'(y)=F(y,a(y))$ with a continuous function $F$ in some region of the $(y,a)$-plane such that the initial value problem has more than one solution in \emph{any} neighborhood of \emph{any} initial point $(y_0,a_0)$ in that region, see \cite{Lavrentieff1925} and/or Section~II.5 in \cite{Hartman2002}.}
Under the assumption of Corollary~\ref{cor:07042020a1} only, the function $a\mapsto F(y,a)$ (for fixed~$y$) need not be locally Lipschitz inside $\cG$; on the contrary, it can have a quite unpleasant behavior (e.g., think about the case when $m$ is singular with respect to the Lebesgue measure). Therefore, standard results on ODEs do not provide uniqueness for~\eqref{eq:07042020a3}.
\end{itemize}

This discussion raises the question of whether there is a numerical algorithm in order to compute the unique solution of the initial value problem~\eqref{eq:07042020a3}.
In what follows, we present such an algorithm.
As in Corollary~\ref{cor:07042020a1}, we assume that $m(\{y\})=0$ for all $y\in I^\circ$. Observe the following facts.
\begin{enumerate}[(a)]
\item\label{it:30042021aa}
$F$ is continuous in $\cG$.
\item\label{it:30042021ab}
$|F(y,a)|<1$ for all $(y,a)\in\cG$.
\item\label{it:30042021ac}
Notice, however, that $F$ is, in general, not extendable to a continuous function on~$\ol\cG$, where $\ol\cG$ denotes the closure of $\cG$ in $\bbR^2$.
\end{enumerate}
By $\partial\cG$ we denote the boundary of $\cG$ in $\bbR^2$:
$\partial\cG=\ol\cG\setminus\cG$.
It will be convenient to decompose $\partial\cG$ as follows:
$$
\partial\cG=\partial\cG_\low\cup\partial\cG_\high\cup\partial\cG_\vertical,
$$
where
\begin{align*}
\partial\cG_\low&=(l_h,r_h)\times\{0\},\\
\partial\cG_\high&=\{(y,a)\in(l_h,r_h)\times(0,\infty):a=(y-l)\wedge(r-y)\},\\
\partial\cG_\vertical&=\partial\cG\setminus(\partial\cG_\low\cup\partial\cG_\high)
\end{align*}
(cf.\ with~\eqref{eq:28042021a4}).
To provide more detail, we remark that, if $l=-r=-\infty$
(hence $l_h=-r_h=-\infty$),
then $\cG=\bbR\times(0,\infty)$, $\partial\cG=\partial\cG_\low$,
$\partial\cG_\high=\partial\cG_\vertical=\emptyset$;
otherwise we also have non-empty $\partial\cG_\high$ and $\partial\cG_\vertical$.
The most interesting case $-\infty<l<r<+\infty$ is illustrated in Figure~\ref{fig:30042021a1}, where we have chosen $l$ to be accessible and $r$ inaccessible in order to show different subcases.

\begin{figure}[!htb]
\begin{tikzpicture}
\draw[->] (-1,0) -- (13,0)
node[below right] {$y$};

\draw[line width=.7mm] (1.5,0)--(1.5,1.5);
\draw (0,0)--(6,6);
\draw[line width=.7mm] (1.5,1.5)--(6,6);
\draw[line width=.7mm] (6,6)--(12,0);
\draw[line width=.7mm] (1.5,0)--(12,0);

\draw (3,4) node{$\partial \mathcal{G}_\text{high}$};
\draw (9,4) node{$\partial \mathcal{G}_\text{high}$};
\draw (2.5,.75) node{$\partial \mathcal{G}_\text{vertical}$};
\draw (0,-0.5) node{$l$};
\draw (1.5,-0.5) node{$l_h$};
\draw (12,-0.5) node{$r_h = r$};
\draw (6,3) node{\Large $\mathcal{G}$};
\draw (6,-0.5) node{$\partial \mathcal{G}_\text{low}$};
\draw[->] (13,1) node[right]{$\partial \mathcal{G}_\text{vertical}$} --(12.1, 0.1) ;
\end{tikzpicture}
{\caption{\footnotesize  Form of the domain $\cG$ and its boundary $\partial\cG$ in the case $-\infty<l<r<+\infty$ with accessible $l$ and inaccessible~$r$ (so that $l_h>l$, $r_h=r$).
}\label{fig:30042021a1}}
\end{figure}
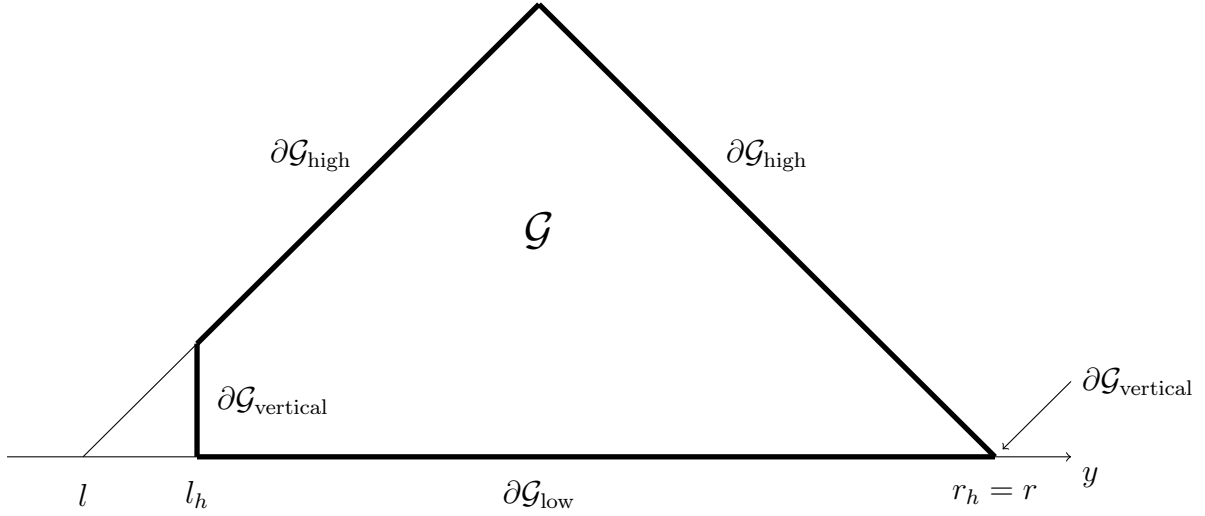

Let us fix a subinterval $(\alpha,\beta)$ of $(l_h,r_h)$ with finite endpoints $\alpha$ and~$\beta$ (we can take $\alpha=l_h$ if $l_h>-\infty$, and similarly with~$\beta$). We now describe how to approximate the unique solution of the initial value problem~\eqref{eq:07042020a3}, which is the EMCEL scale factor $y\mapsto\wh a_h(y)$, uniformly on $[\alpha,\beta]$.
Let $\cY$ be a partition of $[\alpha,\beta]$ that contains $y_0$ from~\eqref{eq:07042020a3} as one of its elements, i.e.,
$\cY=\{y_j\}_{j=K(\cY)}^{L(\cY)}$ with some $K(\cY),L(\cY)\in\bbZ$,
$K(\cY)<0<L(\cY)$ and
$$
\alpha=y_{K(\cY)}<\cdots<y_{-1}<y_0<y_1<\cdots<y_{L(\cY)}=\beta.
$$
We define a piecewise linear function $y\mapsto a^\cY(y)$ on $[y_0,\beta]$ as follows.
\begin{enumerate}
\item
Set $a^\cY(y_0)=\wh a_h(y_0)$ and observe that $(y_0,a^\cY(y_0))\in\cG$. Set $j=0$.

\item\label{it:30042021a2}
Define
$\wt a^\cY(y)=a^\cY(y_j)+(y-y_j)F(y_j,a^\cY(y_j))$ for all $y\in[y_j,y_{j+1}]$.

\item
If $(y_{j+1},\wt a^\cY(y_{j+1}))\in\cG$, then
\begin{itemize}
\item
set $a^\cY=\wt a^\cY$ on $(y_j,y_{j+1}]$,
\item
observe that, by convexity of $\cG$, the graph of $a^\cY$ on $[y_j,y_{j+1}]$ lies inside $\cG$,
\item
set $j=j+1$,
\item
if $j<L(\cY)$, then go to step~\ref{it:30042021a2};
else go to step~\ref{it:30042021a5}.
\end{itemize}

\item
If $(y_{j+1},\wt a^\cY(y_{j+1}))\notin\cG$, then
\begin{itemize}
\item
denote by $(\ol y,\ol a)$ the point, where the graph of $\wt a^{\cY}$ on $[y_j,y_{j+1}]$ intersects $\partial\cG$,
\item
if $\partial\cG_\vertical$ is intersected, which is only possible in the case $\ol y=\beta=r_h$, then set $a^\cY=\wt a^\cY$ on $(y_j,\ol y]$ ($\equiv(y_j,\beta]$) and go to step~\ref{it:30042021a5},
\item
if $\partial\cG_\high$ (resp., $\partial\cG_\low$) is intersected, then set $a^\cY=\wt a^\cY$ on $(y_j,\ol y]$, define $a^\cY$ on $(\ol y,\beta]$ such that its graph goes along $\partial\cG_\high$ on $(\ol y,\beta]$ (resp., set $a^\cY\equiv0$ on $(\ol y,\beta]$) and proceed with step~\ref{it:30042021a5}.
\end{itemize}

\item\label{it:30042021a5}
We constructed the function $[y_0,\beta]\ni y\mapsto a^\cY(y)$.
\end{enumerate}
The function $a^\cY$ is extended to $[\alpha,y_0]$ in a symmetric way.
We thus obtain a piecewise linear function
\begin{equation}\label{eq:30042021a1}
[\alpha,\beta]\ni y\mapsto a^\cY(y),\quad
(y,a^\cY(y))\in\ol\cG,
\end{equation}
which is, in fact, nothing else but Euler's polygonal approximation for~\eqref{eq:07042020a3} suitably extended to $[\alpha,\beta]$.

\begin{lemma}\label{lem:30042021a1}
Assume $m(\{y\})=0$ for all $y\in I^\circ$.
Let $(\alpha,\beta)$ be a subinterval of $(l_h,r_h)$ with finite endpoints $\alpha$ and~$\beta$ (we can take $\alpha=l_h$ if $l_h>-\infty$, and similarly with~$\beta$).
Then, for any sequence $\{\cY_N\}_{N\in\bbN}$ of partitions of $[\alpha,\beta]$, $\cY_N=\{y_j^{(N)}\}_{j=K_N(\cY_N)}^{L_N(\cY_N)}$,
$K_N(\cY_N)<0<L_N(\cY_N)$,
$$
\alpha=y_{K_N(\cY_N)}^{(N)}<\cdots<y_{-1}^{(N)}<y_0^{(N)}=y_0<y_1^{(N)}<\cdots<y_{L_N(\cY_N)}^{(N)}=\beta
$$
with
$$
|\cY_N|:=\max_{K_N(\cY_N)<j\le L_N(\cY_N)}(y_j^{(N)}-y_{j-1}^{(N)})\to0,\quad
N\to\infty,
$$
we have
\begin{equation}\label{eq:30042021a2}
\sup_{y\in[\alpha,\beta]}|a^{\cY_N}(y)-\wh a_h(y)|\to0,\quad
N\to\infty,
\end{equation}
i.e., the sequence $\{a^{\cY_N}\}$ of (suitably extended) Euler's polygonal approximations for~\eqref{eq:07042020a3} converges uniformly on $[\alpha,\beta]$ to the EMCEL scale factor $\wh a_h$, which is the unique solution of~\eqref{eq:07042020a3}.
\end{lemma}

\begin{proof}
The idea is very similar to that in the proof of Peano's existence theorem, where it is shown that, on a sufficiently small interval containing $y_0$, the sequence $\{a^{\cY_N}\}$ has a uniformly convergent subsequence that converges to a solution of~\eqref{eq:07042020a3}.
Peano's existence theorem applies in this form due to fact~\eqref{it:30042021aa} above and the fact that $(y_0,\wh a_h(y_0))\in\cG$.
From Corollary~\ref{cor:07042020a1} we know the existence and uniqueness of solution $\wh a_h$ on the whole $[\alpha,\beta]$ and, moreover, due to the properties of the EMCEL scale factor, $(y,\wh a_h(y))\in\cG$ for all $y\in(\alpha,\beta)$ (even on $(l_h,r_h)$).
We need to make use of that to prove the result and, in particular, explain that the possibility for the approximations $a^{\cY_N}$ to leave the ``good'' region $\cG$ on $[\alpha,\beta]$ (see~\eqref{eq:30042021a1} and recall fact~\eqref{it:30042021ac} above) does not create any problem in our situation.

To this end, we assume that \eqref{eq:30042021a2} does not hold.
Then there exists a small $\eps>0$ and a subsequence
$\{a_1^{(N)}\}$ of $\{a^{\cY_N}\}$ such that, for every~$N$,
\begin{equation}\label{eq:30042021a3}
\sup_{y\in[\alpha,\beta]}|a_1^{(N)}(y)-\wh a_h(y)|\ge\eps.
\end{equation}
By the construction, the sequence $\{a_1^{(N)}\}$ is uniformly bounded and equicontinuous on $[\alpha,\beta]$.
It is worth noting that the uniform boundedness holds even in the case $l=-r=-\infty$
(where $\partial\cG_\high=\emptyset$ and the upper bound is not seen \emph{a priori})
because of fact~\eqref{it:30042021ab} above and the fact that $\alpha$ and $\beta$ are chosen finite.
As for the equicontinuity on the whole $[\alpha,\beta]$, it follows from fact~\eqref{it:30042021ab} together with the fact that the affine functions constituting $\partial\cG_\low$ and $\partial\cG_\high$ have bounded slopes ($0$, $1$ and $-1$). By the Arzel\`a-Ascoli theorem, there exists a uniformly convergent subsequence $\{a_2^{(N)}\}$ of the sequence $\{a_1^{(N)}\}$,
$$
\sup_{y\in[\alpha,\beta]}|a_2^{(N)}(y)-\ol a(y)|\to0,\quad
N\to\infty,
$$
and the limiting continuous function $[\alpha,\beta]\ni y\mapsto\ol a(y)$
is different from $[\alpha,\beta]\ni y\mapsto\wh a_h(y)$
because of~\eqref{eq:30042021a3}.
The standard argumentation in Peano's existence theorem
(see Exercise~II.2.1 in \cite{Hartman2002} or Theorem~2.19 in \cite{Teschl2012})
yields that $\ol a$ is a solution to~\eqref{eq:07042020a3} in a sufficiently small neighborhood of $y_0$.
Hence, by the uniqueness for~\eqref{eq:07042020a3}, $\ol a=\wh a_h$ in a sufficiently small neighborhood of $y_0$. Define
\begin{align*}
\ol\alpha&=\inf\{y\in[\alpha,y_0]:\ol a(z)=\wh a_h(z)\text{ for all }z\in[y,y_0]\}\quad(\in[\alpha,y_0)),\\
\ol\beta&=\sup\{y\in[y_0,\beta]:\ol a(z)=\wh a_h(z)\text{ for all }z\in[y_0,y]\}\quad(\in(y_0,\beta]).
\end{align*}
As $\ol a$ differs from $\wh a_h$, we either have $\ol\alpha>\alpha$ or $\ol\beta<\beta$.
Assume $\ol\beta<\beta$.
But, as $(\ol\beta,\ol a(\ol\beta))=(\ol\beta,\wh a_h(\ol\beta))\in\cG$,
we can again apply the argumentation in Peano's existence theorem and conclude that $\ol a$ is a solution to~\eqref{eq:07042020a3} in a sufficiently small neighborhood of $\ol\beta$, hence $\ol a=\wh a_h$ in that neighborhood. This contradicts the definition of $\ol\beta$ and yields $\ol\beta=\beta$. Similarly, $\ol\alpha=\alpha$. But then $\ol a=\wh a_h$ on the whole $[\alpha,\beta]$, which, in turn, contradicts~\eqref{eq:30042021a3} and thus proves~\eqref{eq:30042021a2}. This concludes the proof.
\end{proof}

In the end we present an example, where we numerically solve the ODE for the EMCEL scale factor and implement our scheme.
It is instructive to do this in a situation of an irregular SDE for $Y$ such that the corresponding Euler scheme does not converge.
We first provide a sufficient condition for that.

\begin{lemma}\label{lem:14052021a1}
Let $y\in I^\circ$.
Consider the setting of Example~\ref{ex:sde}, where the diffusion coefficient $\eta\colon I^\circ\to\bbR\setminus\{0\}$ satisfies\footnote{The notation $y\ne x\to y$ is to emphasize that we consider $\liminf$ as $x$ tends to $y$ over a \emph{deleted} neighborhood of~$y$.}
\begin{equation}\label{eq:14052021a1}
\liminf_{y\ne x\to y}\big[|x-y| |\eta(x)|\big]\in(0,\infty].
\end{equation}
Let $(\zeta_k)_{k\in\bbN}$ be an iid sequence of random variables with $E[\zeta_k]=0$, $E[\zeta_k^2]=1$ and $P(\zeta_k\ne0)>0$.
For each $h\in(0,\ol h]$, we define the Euler scale factor $a_h\colon\bbR\to\bbR$ by\footnote{The extension beyond $I^\circ$ is needed only to treat the possibility of jumping out of the state space.}
$a_h(x)=\eta(x)\sqrt h$, $x\in I^\circ$, and $a_h(x)=0$, $x\in\bbR\setminus I^\circ$,
and define the (linearly interpolated, generalized) Euler scheme $X^{h,y}=(X^{h,y}_t)_{t\in[0,\infty)}$ through $a_h$ and $(\zeta_k)$
in the same way as the EMCEL scheme $\wh X^{h,y}$ is defined through $\wh a_h$ and $(\xi_k)$ in \eqref{eq:def_X}--\eqref{eq:13112017a1}.
Then, for any $T\in(0,\infty)$ and for any sequence $\{h_n\}_{n\in\bbN}\subset(0,\infty)$ with $h_n\to0$, the sequence of the laws of the processes $(X^{h_n,y}_t)_{t\in[0,T]}$ does not converge weakly in $C([0,T],\bbR)$.
\end{lemma}

Essentially, this result says that Euler-type schemes\footnote{I.e., the schemes corresponding to different distributions of $\zeta_k$.} fail to converge if \eqref{eq:14052021a1} holds.
Notice that \eqref{eq:14052021a1} is satisfied whenever the diffusion coefficient $\eta$ has in a deleted neighborhood of the point $y$ a singularity proportional to $|x-y|^{-\alpha}$, for some $\alpha\ge1$.

\begin{remark}\label{rem:14052021a1}
(i) We emphasize that not only the convergence of the distributions of $(X^{h,y}_t)$ to that of $(Y_t)$ fails;
the distributions of $(X^{h,y}_t)$ do not converge at all.

\smallskip
(ii) Neither it is possible to extract a convergent subsequence (for some $\{h_n\}\subset(0,\infty)$ with $h_n\to0$).

\smallskip
(iii) Moreover, the distributions fail to converge on arbitrarily small intervals $[0,T]$ (with $T\in(0,\infty)$).
\end{remark}

\begin{proof}[Proof of Lemma~\protect\ref{lem:14052021a1}]
For each $h\in(0,\infty)$, we have $X^{h,y}_0=y$, $X^{h,y}_h=y+\eta(y)\sqrt h\,\zeta_1$ and hence
\begin{equation}\label{eq:16052021a1}
X^{h,y}_{2h}=y+\eta(y)\sqrt h\,\zeta_1
+\eta\left(y+\eta(y)\sqrt h\,\zeta_1\right)\sqrt h\,\zeta_2.
\end{equation}
We fix some $T\in(0,\infty)$ and consider, for all $\eps\in(0,T]$, bounded continuous path functionals $\Phi^\eps\colon C([0,T],\bbR)\to\bbR$ defined by the formula
$$
\Phi^\eps(\omega)=\sup_{s\in[0,\eps]}|\omega(s)-y|\wedge1, \qquad \omega \in C([0,T],\bbR).
$$
It follows from~\eqref{eq:16052021a1} that, for $h\in(0,\frac\eps2]$,
$$
\Phi^\eps(X^{h,y})\ge\left|\eta(y)\sqrt h\,\zeta_1
+\eta\left(y+\eta(y)\sqrt h\,\zeta_1\right)\sqrt h\,\zeta_2\right|\wedge1.
$$
Now take an arbitrary sequence $\{h_n\}_{n\in\bbN}\subset(0,\infty)$ with $h_n\to0$. Denoting by $A\in(0,\infty]$ the $\liminf$ in~\eqref{eq:14052021a1} we obtain by the Fatou lemma that
\begin{align}
\liminf_{n\to\infty}E[\Phi^\eps(X^{h_n,y})]
&\ge
E\left[\liminf_{n\to\infty}\left|\eta(y)\sqrt{h_n}\,\zeta_1
+\eta\left(y+\eta(y)\sqrt{h_n}\,\zeta_1\right)\sqrt{h_n}\,\zeta_2\right|\wedge1\right]
\notag\\
&\ge
E\left[1_{\{\zeta_1\ne0\}}\left(\left|\frac{A\zeta_2}{\eta(y)\zeta_1}\right|\wedge1\right)\right]>0.
\label{eq:16052021a2}
\end{align}
Assume that the sequence of the processes $(X^{h_n,y}_t)_{t\in[0,T]}$ converges to some process $(X^y_t)_{t\in[0,T]}$ in distribution.
Then $(X^y_t)_{t\in[0,T]}$ is a continuous process starting in~$y$.
By the dominated convergence theorem it follows
$$
\lim_{\eps\to0}E[\Phi^\eps(X^y)]=0.
$$
This and~\eqref{eq:16052021a2} contradict the weak convergence of the laws of $(X^{h_n,y}_t)_{t\in[0,T]}$ to the law of $(X^y_t)_{t\in[0,T]}$.
This completes the proof.
\end{proof}

Finally, we illustrate the usefulness of the ODE~\eqref{eq:07042020a3} as well as the implications of Proposition~\ref{prop:08022020a1} and Lemma~\ref{lem:14052021a1} in a numerical experiment.

\begin{ex}\label{ex:singular_eta}
Let $b\in (0,\infty)$ and $I=\bbR$.
Define the function $\eta\colon \bbR \to (0,\infty)$ by the formula
\begin{equation}\label{eq:singular_eta}
\eta(x)=1_{\bbR\setminus\{0,b\}}(x)\frac{1}{|x||x-b|}+1_{\{0,b\}}(x),\quad x\in\bbR,
\end{equation}
and let the process $Y$ solve\footnote{Note that $\eta$ satisfies \eqref{eq:27092018a2} and~\eqref{eq:27092018a3}, i.e., we are in the setting of Example~\ref{ex:sde}.} $dY_t=\eta(Y_t)\,dW_t$, $Y_0=x_0\in \bbR$. Observe that $\eta$ satisfies~\eqref{eq:14052021a1} for 
$y\in \{0,b\}$ and thus Lemma~\ref{lem:14052021a1} implies that Euler-type schemes fail to converge in the case $x_0\in \{0,b\}$. By Proposition~\ref{prop:07022020a2} the EMCEL scheme converges weakly to $Y$ for all initial values $x_0\in \bbR$ and Corollary~\ref{cor:07042020a1} ensures that for every $h\in (0,\ol h]$ the EMCEL scale factor $\wh a_h \in C^1(\bbR, \bbR)$ is the unique solution of the initial value problem~\eqref{eq:07042020a3}.

Figure~\ref{fig_singular_sde_sf} depicts numerical solutions $\wh a_h$ (top) of this ODE as well as normalized solutions $\wh a_h/\sqrt{h}$ (middle) and solutions relative to the Euler scale factors $\wh a_h/(\sqrt{h}\eta)$ (bottom) in the case $b=2$ for $h\in 2\cdot \{10^{-1},10^{-2},10^{-3},10^{-4}\}$. In particular, the plot at the bottom indicates that $\lim_{h\to 0}\wh a_h(y)/(\sqrt{h}\eta(y))=1$ for all $y\in \bbR\setminus\{0,b\}$ as established in Proposition~\ref{prop:08022020a1}. For $y\in \{0,b\}$ the plot in the middle suggests that $\wh a_h(y)$ goes slower to $0$ than $\sqrt{h}$ which is confirmed by Proposition~\ref{prop:08022020a1} ensuring that $\lim_{h\to 0}\wh a_h(y)/(h^{1/4})=(6/b^2)^{1/4}$ in this case.

\begin{figure}[!htb]
\centering
\includegraphics[width=0.5\textwidth]{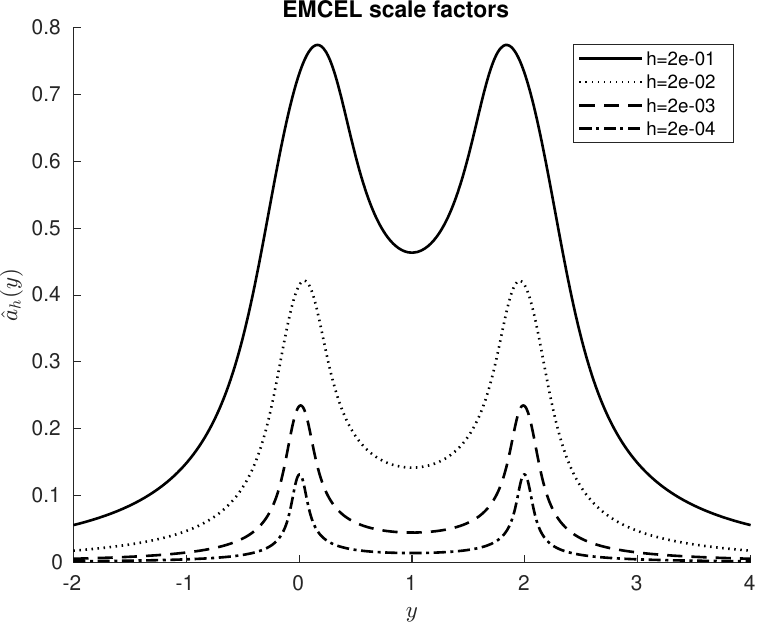}\\
\includegraphics[width=0.5\textwidth]{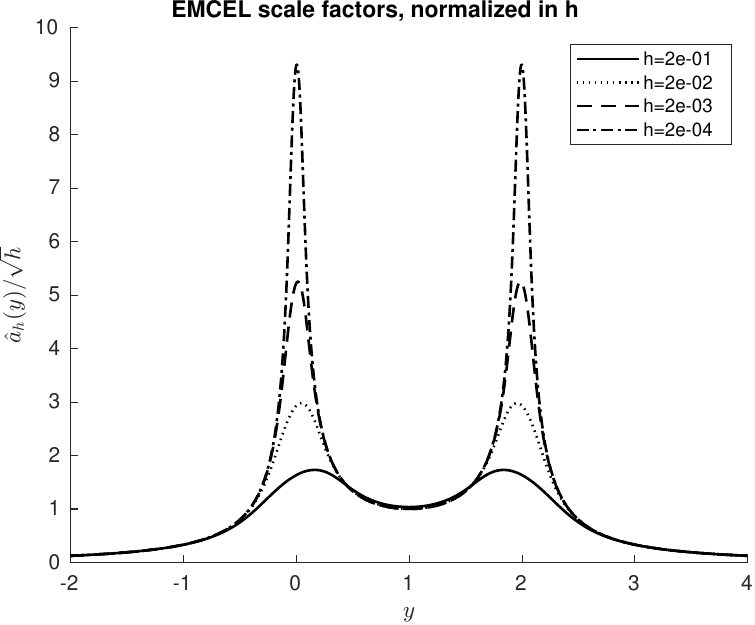}\\
\includegraphics[width=0.5\textwidth]{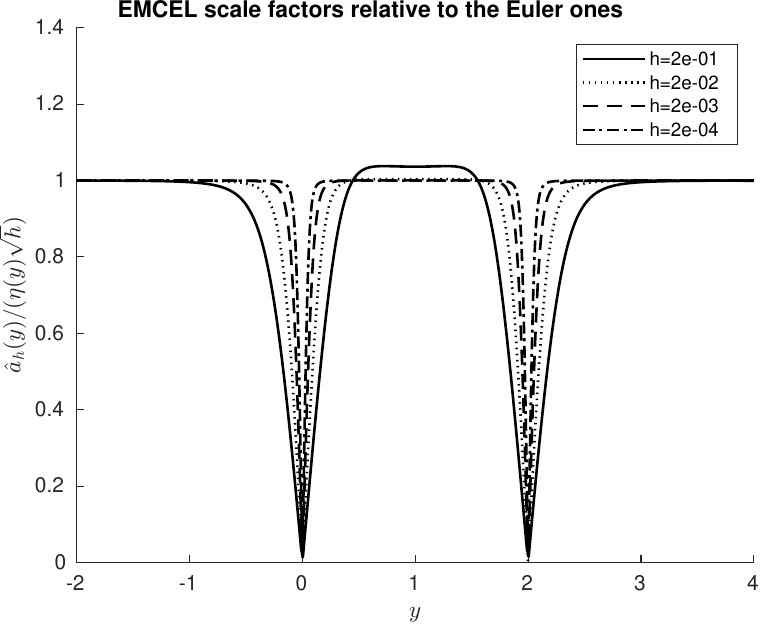}

{\caption{\footnotesize  
The figure shows solutions $\wh a_h$ (top), normalized solutions $\wh a_h/\sqrt{h}$ (middle) and solutions relative to the Euler scale factors $\wh a_h/(\sqrt{h}\eta)$ (bottom) of the initial value problem~\eqref{eq:07042020a3} associated to~\eqref{eq:singular_eta} for the case $b=2$ and for $h\in 2\cdot \{10^{-1},10^{-2},10^{-3},10^{-4}\}$.
}\label{fig_singular_sde_sf}}
\end{figure}

Figure~\ref{fig_singular_sde_perf} illustrates the performance of the weak Euler scheme\footnote{See the text preceding Corollary~\ref{cor:06022020a1} for the definition of the weak Euler scheme.} and of the EMCEL scheme in this setting. The upper panel shows two realizations (dotted: weak Euler, solid: EMCEL) with the same random increments $(\xi_k)$ in the case $x_0=1$, $h=2\cdot 10^{-3}$ with time horizon $T=2$. As long as the trajectories move around the level~$1$ one can barely spot a difference between the trajectories. But as they approach the singularity of $\eta$ at~$0$ around time $0.6$ the weak Euler scheme is shot to a level bigger than~$5$, whereas the EMCEL scheme continues smoothly only exhibiting a higher volatility. The reason is that, in contrast to the EMCEL scale factor, the Euler scale factor inherits the singularities of $\eta$. 
This observation is also manifested in the lower panel of Figure~\ref{fig_singular_sde_perf} which depicts the empirical distribution functions of $X^{Eu,h}_2$ (dotted) and $\wh X^{h}_2$ (solid) in this case for a sample size $M=10^5$. One sees that, opposed to the EMCEL scheme, the weak Euler scheme puts a lot of mass outside the interval $[-2,4]$ indicating that the weak Euler scheme exhibits the behavior portrayed in the upper panel of Figure~\ref{fig_singular_sde_perf} with high probability.

\begin{figure}[!htb]
\centering
\includegraphics[width=0.5\textwidth]{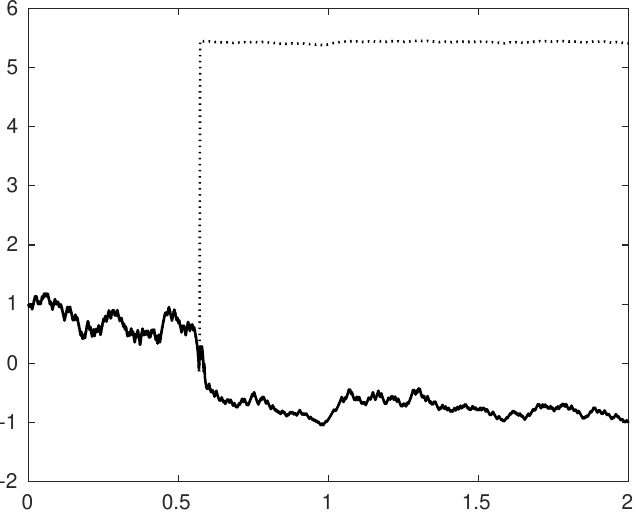}\\
\includegraphics[width=0.5\textwidth]{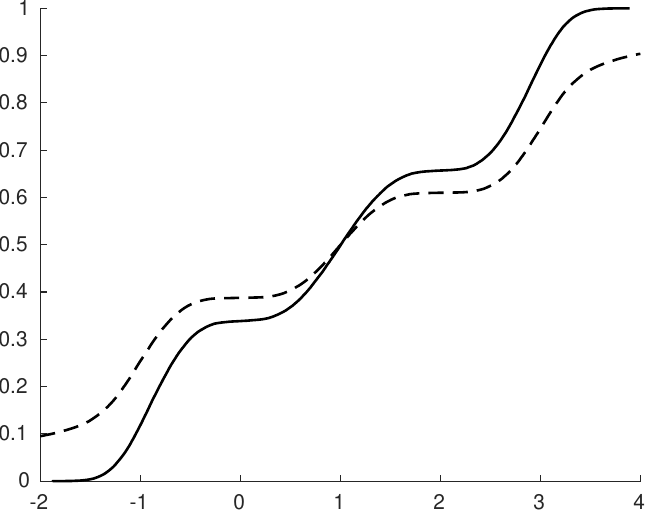}

{\caption{\footnotesize  
The upper panel shows trajectories of the weak Euler scheme $X^{Eu,h}$ (dotted) and the EMCEL scheme $\wh X^h$ (solid) in the setting of Example~\ref{ex:singular_eta} in the case $b=2$, $x_0=1$, $h=2\cdot 10^{-3}$ with time horizon $T=2$. The lower panel shows the empirical distribution functions of $X^{Eu,h}_2$ (dotted) and $\wh X^h_2$ (solid) in this case for a sample size $M=10^5$.
}\label{fig_singular_sde_perf}}
\end{figure}
\end{ex}

\paragraph{Acknowledgement}
We thank three anonymous referees for their constructive comments and suggestions that helped us improve the manuscript.
Wolfgang L\"ohr and Mikhail Urusov acknowledge the support from the
\emph{German Research Foundation} through the project 415705084.

\bibliographystyle{abbrv}
\bibliography{literature}

\end{document}